\documentclass[11pt]{article}
\usepackage[T1]{fontenc}
\usepackage{graphicx, subcaption}
\usepackage[pdftex,dvipsnames,table]{xcolor}
\usepackage{amsmath, amsthm, amssymb}
\usepackage{bbold}
\usepackage{booktabs, tabularx, multirow}
\usepackage[nocomma, c3]{optidef}
\usepackage{cleveref}
\usepackage{todonotes}
\usepackage{xargs}
\usepackage[margin=1in]{geometry}
\usepackage{tcolorbox}
\usepackage{xspace}
\usepackage{float}
\usepackage{url}
\usepackage{tikz, pgfplots}
\usetikzlibrary{shapes, arrows}
\usepackage{siunitx}
\usepackage{enumitem}
\usepackage{authblk}

\usepackage{setspace}
\usepackage[round]{natbib}
\usepackage{endnotes}
\let\footnote=\endnote

\title{Solving the Line-Based Dial-a-Ride Problem\\by Generating Stopping Patterns}
\author[1]{Antonio Lauerbach}
\author[2,3]{Sven Mallach}
\author[1]{Kendra Reiter}
\author[1]{Marie Schmidt}
\author[4]{Michael Stiglmayr}
\affil[1]{Department of Computer Science, University of Würzburg, Germany, \texttt{\{firstname.lastname\}@uni-wuerzburg.de}}
\affil[2]{Network \& Data Science Management Group, University of Siegen, Germany}
\affil[3]{University of Bonn, Germany, \texttt{sven.mallach@cs.uni-bonn.de}}
\affil[4]{School of Mathematics and Natural Sciences, University of Wuppertal, Germany, \texttt{stiglmayr@uni-wuppertal.de}}
\date{\vspace*{-2em}}

\newcommand{\dir}{\mathrm{dir}}
\newcommand{\asc}{\mathrm{asc}}
\newcommand{\desc}{\mathrm{desc}}

\newcommand{\Sset}{\mathcal{S}}
\newcommand{\Rset}{\mathcal{R}}
\newcommand{\Pset}{\mathcal{P}}
\newcommand{\Hset}{\mathcal{H}}
\newcommand{\Kset}{\mathcal{K}}
\newcommand{\Jset}{\mathcal{J}}
\newcommand{\Gset}{\mathcal{G}}

\newcommand{\Vset}{\mathcal{V}}
\newcommand{\Iset}{\mathcal{I}}

\newcommand{\Rasc}{\Rset^{\asc}}
\newcommand{\Rdesc}{\Rset^{\desc}}
\newcommand{\Pasc}{\Pset^{\asc}}
\newcommand{\Pdesc}{\Pset^{\desc}}
\newcommand{\Gasc}{\Gamma^{\asc}}
\newcommand{\Gdesc}{\Gamma^{\desc}}
\newcommand{\Gdir}{\Gamma^{\dir}}
\newcommand{\Rdir}{\Rset^{\dir}}
\newcommand{\Pdir}{\Pset^{\dir}}
\newcommand{\dirs}{\dir \in \{\asc, \desc\}}
\newcommand{\Qmax}{Q}
\newcommand{\Sstart}[2]{\mathrm{START}^{#1}_{#2}}
\newcommand{\Send}[2]{\mathrm{END}^{#1}_{#2}}
\newcommand{\Right}[2]{\mathrm{R}^{#1}_{#2}}
\newcommand{\Left}[2]{\mathrm{L}^{#1}_{#2}}

\newcommand{\sizeS}{u}
\newcommand{\sizeR}{m}
\newcommand{\sizeH}{n}
\newcommand{\sizeK}{c}
\newcommand{\sizeP}{q}
\newcommand{\wpax}{w_{\text{pax}}}
\newcommand{\wdist}{w_{\text{dist}}}

\newcommand{\NP}{\ensuremath{\mathsf{NP}}\xspace}

\newcommand{\Pclique}{\textsc{clique}\xspace}
\newcommand{\Plidarp}{\textsc{liDARP}\xspace}
\newcommand{\Plidarplong}{\textsc{line-based dial-a-ride problem}\xspace}

\newcommand{\Ppattern}{\textsc{most profitable stopping pattern}\xspace}
\newcommand{\Ppatternuncap}{\textsc{most profitable uncapacitated stopping pattern}\xspace}
\newcommand{\lipdp}{\textsc{liDARP without TWs}\xspace}
\newcommand{\darp}{\textsc{DARP}\xspace}
\newcommand{\pdp}{\textsc{PDP}\xspace}
\newcommand{\vrp}{\textsc{VRP}\xspace}

\newcommand{\defproblem}[6]{
  \begin{tcolorbox}[colback=#5!5!white,colframe=#5!75!black]%
    \hspace{0ex}\hspace*{-2.8ex}
    \begin{minipage}{0.99\textwidth}
      \vspace{0ex}\vspace*{-1ex}
      {\sf\bfseries\color{#5!75!black} #6 Problem:} #1\\[.1ex]
      \begin{tabular}{@{}l@{~~}p{0.9\textwidth}@{}}
        {\sf\bfseries\color{#5!75!black} Input:} & #2\\[.1ex]
        {\sf\bfseries\color{#5!75!black} #4:} & #3\\[-1ex]
      \end{tabular}
    \end{minipage}
  \end{tcolorbox}
}
\newcommand{\defdecproblem}[3]{\defproblem{#1}{#2}{#3}{Question}{gray}{Decision}}
\newcommand{\defoptproblem}[3]{\defproblem{#1}{#2}{#3}{Output}{gray}{Optimization}}

\newtheorem{thm}{Theorem}

\usepackage{mathtools}
\DeclarePairedDelimiter\set{\{}{\}} 

\def\Oh{\ensuremath{\mathcal{O}}} 

\definecolor{PKdarkblue}{rgb}{0.121,0.47,0.705}
\definecolor{PKdarkred}{rgb}{0.89 0.102 0.109}
\definecolor{PKdarkgreen}{rgb}{0.2 0.627 0.172}
\definecolor{PKdarkorange}{rgb}{1 0.498 0}
\definecolor{PKdarkpurple}{rgb}{0.415 0.239 0.603}
\definecolor{PKdarkpink}{rgb}{0.969 0.506 0.749}
\definecolor{PKdarkyellow}{rgb}{1 1 0.2}
\definecolor{PKlightgray}{rgb}{0.8 0.8 0.8}
\definecolor{PKdarkgray}{rgb}{0.5 0.5 0.5}
\definecolor{PKlightblue}{rgb}{0.651 0.807 0.89}
\definecolor{PKlightgreen}{rgb}{0.698 0.874 0.541}
\definecolor{PKlightorange}{rgb}{0.992 0.749 0.435}

\usepgfplotslibrary{statistics}
\pgfplotsset{compat=1.18,small,scale only axis,enlarge x limits={value=0.03,auto},
layers/marklayers/.define layer set={
    bg,main,m1,m2,m3,m4
}{},set layers=marklayers}
\pgfplotscreateplotcyclelist{scatter-marks-2-plots-light-dark}{
  {solid,mark options={scale=.95,fill=PKlightblue,line width=1pt},mark=square*,PKlightblue},
  {solid,mark options={scale=.95,fill=PKdarkblue,line width=0.7pt},mark=Mercedes star,PKdarkblue},
  {solid,mark options={scale=.95,fill=PKlightorange,line width=1pt},mark=square,PKlightorange},
  {solid,mark options={scale=.95,fill=PKdarkorange,line width=0.7pt},mark=Mercedes star flipped,PKdarkorange}
}
\pgfplotscreateplotcyclelist{scatter-marks-4-plots}{
  {solid,mark options={scale=.95,fill=PKdarkgreen,line width=0.7pt},mark=o,PKdarkgreen},
  {solid,mark options={scale=.95,fill=PKdarkred,line width=0.7pt},mark=square,PKdarkred},
  {solid,mark options={scale=.95,fill=PKdarkblue,line width=0.7pt},mark=diamond,PKdarkblue},
  {solid,mark options={scale=.95,fill=PKdarkorange,line width=0.7pt},mark=triangle,PKdarkorange}
}
\pgfplotscreateplotcyclelist{scatter-marks-6-plots}{
  {solid,mark options={scale=.95,fill=PKdarkgreen,line width=0.7pt},mark=o,PKdarkgreen},
  {solid,mark options={scale=.95,fill=PKdarkred,line width=0.7pt},mark=square,PKdarkred},
  {solid,mark options={scale=.95,fill=PKdarkblue,line width=0.7pt},mark=diamond,PKdarkblue},
  {solid,mark options={scale=.95,fill=PKdarkorange,line width=0.7pt},mark=triangle,PKdarkorange},
  {solid,mark options={scale=.95,fill=PKdarkpurple,line width=0.7pt},mark=pentagon,PKdarkpurple},
  {solid,mark options={scale=.95,fill=PKdarkpink,line width=0.7pt},mark=star,PKdarkpink}
}
\pgfplotscreateplotcyclelist{scatter-marks-9-plots}{
  {solid,mark options={scale=.95,fill=PKdarkgreen,line width=0.7pt},mark=o,PKdarkgreen},
  {solid,mark options={scale=.95,fill=PKdarkred,line width=0.7pt},mark=square,PKdarkred},
  {solid,mark options={scale=.95,fill=PKdarkblue,line width=0.7pt},mark=diamond,PKdarkblue},
  {solid,mark options={scale=.95,fill=PKdarkorange,line width=0.7pt},mark=triangle,PKdarkorange},
  {solid,mark options={scale=.95,fill=PKdarkpurple,line width=0.7pt},mark=pentagon,PKdarkpurple},
  {dashed,mark options={scale=.95,fill=PKdarkgreen,line width=1pt},mark=*,PKdarkgreen},
  {dashed,mark options={scale=.95,fill=PKdarkred,line width=1pt},mark=square*,PKdarkred},
  {dashed,mark options={scale=.95,fill=PKdarkblue,line width=1pt},mark=diamond*,PKdarkblue},
  {dashed,mark options={scale=.95,fill=PKdarkorange,line width=1pt},mark=triangle*,PKdarkorange},
}
\pgfplotscreateplotcyclelist{boxplot-colors}{
    {solid,PKdarkgreen}, 
    {solid,PKdarkred}, 
    {solid,PKdarkblue}, 
    {solid,PKdarkorange}
}

\begin{document}
\renewcommand{\arraystretch}{1.2}

\maketitle

\begin{abstract}
    In the \Plidarplong (\Plidarp), vehicles operate along a predefined bus line, with the possibility of skipping stations and turning when empty. 
    Motivated by the practical observation that tight passenger time windows often limit pooling in on-demand services, we introduce a new variant of this transportation system by removing all temporal constraints, which we call the \lipdp.

    We introduce a new MILP formulation for the \lipdp, which constructs feasible tours as sequences of stopping patterns; first, we consider a fundamental single-vehicle, single-pass special case. Based on our insights, we develop a branch-and-price algorithm where the pricing problem generates profitable stopping patterns.
    For practical applications, we additionally propose a root node heuristic, using the stopping patterns generated at the root node.

    Computational experiments show that our branch-and-price algorithm is competitive, finding solutions with a MIP gap of less than 5\% for large instances in 60 minutes. Further, the root node heuristic scales to instances with up to 100 requests, outperforming the state-of-the-art and reaching optimality gaps of less than 5\% within 15 minutes.
    This method is highly effective in generating solutions for practical applications, where solving large problems quickly is more valuable than reaching optimality.
\end{abstract}


\section{Introduction}\label{sec:intro}

Public transportation systems can be categorized based on their flexibility \citep[see][]{errico_survey_2013}: classical scheduled systems are \emph{fixed} in both spatial and temporal dimensions, with regular tours and timetables (or headways). In contrast, on-demand systems are \emph{fully-flexible}, where vehicles may take any path between customers at any time. The optimization problem associated with finding the best tour (most commonly: minimizing driving costs) for a fleet of vehicles serving a specific set of requests in the fully-flexible setting is called the \textsc{dial-a-ride problem} (\darp). Transportation paradigms which combine characteristics of both the traditional public transport and on-demand systems are called \emph{semi-flexible}.

One such semi-flexible system is known in German as \textsc{Bedarfslinienbetrieb} \citep{mehlert1998angebotsbezeichnungen} (literally: demand-responsive line-based operation), for which, to the best of our knowledge, no exact English equivalent exists in the traffic engineering literature. This system adds spatial flexibility to an existing bus line by allowing that a vehicle's tour may visit only a subset of stations based on the pre-booked passenger requests. In particular, the vehicle might use shortcuts or turn before the final stop of the line. The systems employ a conventional timetable, which,
according to \cite{siefer_handbuch_2023} may lead to higher pooling rates compared to more temporally flexible systems. To the best of our knowledge, this increase in pooling rates has not been empirically confirmed. Examples include the Rufbus Eiderstedt\footnote{\url{https://www.nordfriesland.de/Wirtschaft-Tourismus/OPNV/Rufbus-Eiderstedt/}}, the StadtBus Melsungen\footnote{\url{https://www.melsungen.de/stadtbuslinie-und-anrufsammeltaxi/}}, and the AnrufbusFlex\footnote{\url{https://www.mein-bus.net/Linienverkehr/Landkreis-Wittenberg/Anrufbus/index.html}}.

We consider a variation of this transportation system in which we further introduce temporal flexibility by removing the timetable, whilst preserving the spatial structure (and flexibility) of the underlying bus line. Vehicles may skip stations, 
but may only turn around when they are empty, to ensure that every request is always transported in their original direction of travel (with respect to the bus line). This variation has, to the best of our knowledge, not yet been implemented or studied in practice. We call the optimization problem of finding optimal tours in this setting the \Plidarplong (\Plidarp), as a variation of the \darp. The \Plidarp was first described by \cite{reiter_line-based_2024}.

Similar to the assumptions made in the \darp, in the \Plidarp, we determine which requests to serve when and by which vehicle, or to decline a request, while strictly adhering to capacity restrictions. In this aspect, the \Plidarp modeling differs significantly from modeling traditional public transport, where passengers are often assumed to board on a first-come, first-served basis, and capacity restrictions are generally not enforced.

In the \darp, each request is associated with a time window (whose length corresponds to the maximum waiting time) and a maximum ride time, which must be respected by a feasible solution. Additionally, each vehicle must adhere to a maximum duration of service for its operation. 
In the literature on the \darp, a maximum waiting time of 15 minutes has been established by the benchmark instances of~\cite{cordeau_branch-and-cut_2006} and~\cite{ropke_models_2007}, which was adopted for the \Plidarp in~\cite{reiter_line-based_2024}. 
In reality, the waiting time of \darp systems varies widely, reaching mean values of 30 minutes in Germany~\citep{vdv_linien-_2025}. Furthermore, allowing for a longer maximum waiting time during times of disproportionately high demand may reduce the time commitment of the vehicle fleet, increasing flexibility~\citep{ siefer_handbuch_2023,agora_verkehrswende_mobilitatsoffensive_2023}.
Existing state-of-the-art solution methods for the \darp and the \Plidarp exploit the tight time windows in their pre-processing to reduce the number of compatible passenger requests.  
In particular,~\cite{gaul25tight} observe that the solution time for their event-based formulation for the \darp increases significantly with the length of the time window. This motives the development of new methodological approaches.

In this paper, we consider a variant of the \Plidarp by removing the passenger time windows entirely, which we call the \lipdp. We first study the \Ppattern problem, a special case where a single vehicle travels along the line once. We show that the \Ppattern problem is already \NP-hard and propose an integer linear programming (ILP) formulation, by interpreting it as a specific path-finding problem in an acyclic tournament digraph. 
To solve the general \lipdp, we propose a new mixed-integer linear programming (MILP) formulation that builds vehicle tours as a sequence of so-called stopping patterns. 
Since the number of potential stopping patterns is exponential in the number of stations, we develop a branch-and-price algorithm, i.e., we iteratively generate promising stopping patterns by solving a variant of the \Ppattern problem. 
While our branch-and-price algorithm is guaranteed to find optimal solution eventually, it may be too slow for practical purposes. Therefore, as a heuristic alternative, we propose to solve the restricted master problem based on the columns generated in the root node.

The remainder of the paper is structured as follows: we start by discussing related work (\Cref{{sec:intro:lit}}) and give a formal definition of the \lipdp (\Cref{sec:intro:prob}). In \Cref{sec:model}, we discuss the mathematical modeling details and present an MILP model. We introduce the related \Ppattern problem (\Cref{sec:gen-patterns}), show that it is \NP-hard (\Cref{sec:gen-patterns:complexity}), and present solution methods (\Cref{sec:gen-patterns:solving}). The branch-and-price algorithm is presented in detail in \Cref{sec:bap}. Finally, in \Cref{sec:experiments}, computational results are given, where we compare to the state-of-the-art method and investigate our proposed method's performance under various settings.

\subsection{Related Work}\label{sec:intro:lit}
The \Plidarp is first introduced by~\cite{reiter_line-based_2024}, wherein the authors propose and compare three MILP formulations for the \Plidarp: (i) the \emph{location-based} formulation, modifying the three-index formulation for the \darp by~\cite{cordeau_branch-and-cut_2006}, (ii) the \emph{event-based} formulation, built upon the event-based graph introduced by~\cite{gaul_solving_2021}, and (iii) a \emph{subline-based} formulation that models sublines as flows in a network (and, in this, crucially differs from the formulation proposed here). 
In the \Plidarp, the objective is to maximize the weighted sum of accepted requests and the \emph{saved distance}, i.e., the difference between total driven distance and the booked request distance.
\cite{barth_line-based_2025} propose and study
a variant of the \Plidarp, which allows for transfers between multiple lines, called the \textsc{line-based dial-a-ride problem with transfers}.

Research by~\cite{lauerbach_complexity_2025} investigates the complexity of the \Plidarp and the related \textsc{MinTurn} problem. The authors show that, if the number of transported requests is to be maximized (neglecting the driving distance entirely), the \lipdp is solvable in polynomial time. The proof relies on the idea of decomposing the vehicle's tour into sublines, defining a sequence of stops between turns. Other problem variants are shown to be \NP-hard. 

While research on the \Plidarp is still limited, a wide range of related optimization problems has been studied, which the following overview attempts to capture.

First, the \Plidarp can be understood as one variant within the broad class of \textsc{vehicle routing problems} (\vrp). 
\vrp that include precedence (a good or passenger must be picked up before they are dropped off) and pairing (both locations must be visited by the same vehicle) constraints are often referred to as \textsc{pickup and delivery problems} (\pdp), where goods are concerned, or \textsc{dial-a-ride problems} (\darp), where passengers are concerned. 
An overview and classification of variants of the \pdp can be found in~\cite{berbeglia_static_2007} and~\cite{parragh_survey_2008}. For the \darp, we refer to~\cite{cordeau_dial--ride_2007} for an overview up to 2007, and to~\cite{molenbruch_typology_2017} or~\cite{ho_survey_2018} until 2018. 
Most state-of-the art solution approaches~\citep[compare][]{cordeau_branch-and-cut_2006,gaul_event-based_2022} exploit the existence of (ideally, tight) time windows, while there are few exact approaches to solve the \pdp or \darp without time windows: \cite{ruland_pickup_1997} provide an integer formulation of the \pdp with a single vehicle and propose a branch-and-cut algorithm to solve instances of up to 15 requests. \cite{kalantari_algorithm_1985} modify a branch-and-bound algorithm to solve the \pdp with both single and multiple vehicles, with and without capacities. 

\cite{lysgaard_pyramidal_2010} introduce a restricted version of the \vrp, called the \textsc{pyramidal capacitated vehicle routing problem} (\textsc{PCVRP}), and propose a branch-and-cut-and-price algorithm. In the \textsc{PCVRP}, each feasible solution is \emph{pyramidal}: a sequence $(0, i_1, \ldots, i_\alpha, j_\beta, \ldots, j_1, 0)$ starting and ending at the depot $0$, visiting $\alpha + \beta$ request locations at nodes $i_1, \ldots, i_\alpha, j_1, \ldots, j_\beta$ with $i_1 < \ldots < i_\alpha$ and $j_1 < \ldots < j_\beta$. This problem has a similar structure to the \Plidarp, where the order of visiting customers is imposed by the underlying line, though the \textsc{PCVRP} has no precedence or pairing constraints and only one turn in its tour. 

If, in the single vehicle case, nodes are permitted to be both the origin and destination of possibly several pickup and delivery requests, one obtains the \textsc{asymmetric traveling salesman problem with precedence constraints}~\citep{Ascheuer2000}. Its path-variant is known as the \textsc{sequential ordering problem}; both problems have also been studied in the presence of capacities~\citep{HERNANDEZPEREZ2009,LETCHFORDSG2016}.

In the \textsc{target visitation problem}~\citep{Hildenbrandt2019, mallach_refined_2025}, instead of being given a fixed line, the task is to determine an optimal ordering of stations to maximize the difference between a sum of pair-based profits and the total driving distance, in the setting of one vehicle without capacity restrictions.
For our pricing problem presented later on, we will also be confronted with the task of determining a source-to-sink path maximizing the difference between the sum of rewards obtained from realized pickup and delivery requests and the total driving cost, without the necessity to visit all the nodes. As mentioned in the introduction, we refer to this problem as \Ppattern.
A related problem is the \textsc{maximum required pairs with single path problem} studied and shown to be \NP-hard by~\cite{MaxRPSP2014}. This problem, however, does not consider the cost part and acts on a directed acyclic graph that is not necessarily complete, while our subproblem is settled on an acyclic tournament digraph.

The \textsc{capacitated profitable tour problem} (\textsc{CPTP}) introduced by~\cite{archetti_capacitated_2009}, a variation of the well-studied \textsc{team orienteering problem}, considers a similar objective function as the \Plidarp, where the difference between the total profit (awarded per accepted request) and the routing costs is maximized. 
Exact solution methods based on branch-and-price to the \textsc{CPTP} are discussed by~\cite{archetti_optimal_2013}, where the pricing problem is an \textsc{elementary shortest path problem with resource constraints}. 
In this problem, the vehicles are fully flexible and may take any path between nodes, without precedence or pairing constraints.

We emphasize that the \Plidarp and \lipdp differ from the here-mentioned problems in the spatial rigidity through the underlying line structure to be adhered to.

In planning fixed public transportation systems, the determination of lines (\textsc{line planning}), the timing of these (\textsc{timetabling}), and the assignment of vehicles (\textsc{rolling stock scheduling}), which are all intrinsic decisions in the \Plidarp, are traditionally treated as three separate optimization problems and solved sequentially. This is due to the differing time horizons at which each step must be taken, and the high complexity of each step. The use of column generation to solve these problems has been investigated, e.g., for the \textsc{line planning} problem~\citep{borndorfer_column-generation_2007, gatt_solving_2025} or the \textsc{timetabling} problem~\citep{cacchiani_column_2008}. Further research has been done on integrated systems, combining multiple planning steps, e.g., by~\cite{schobel_eigenmodel_2017} and~\cite{liu_optimizing_2023}.

Demand in fixed public transportation systems is ordinarily considered as aggregated origin-destination pairs and is not known explicitly. In such a setting, the usage of express lines, which have optional and fixed stops and operate alternatively to the entire line has been considered by~\cite{roth_energy-efficient_2025}. \cite{gkiotsalitis_subline_2022} use a pre-determined candidate set of express lines, named sublines, to investigate the optimal operational frequency under uncertain demand.

\subsection{Problem Definition}\label{sec:intro:prob}
In this section, we formally introduce the \lipdp. Note that, depending on the precise application, it may be reasonable to assume that all distances satisfy the triangle inequality. 
In this light, we would like to emphasize that our complexity results do extend to the case of metric distances. However, our proposed solution methods do not require the distances to be metric, whence this property is not imposed as a general assumption in our problem definition.

We consider a sequence of stations $\Hset := (1, \ldots, \sizeH)$ with pairwise, symmetric distances $t_{h, h'} \ge 0$ for $h, h' \in \Hset$. The sequence \(\Hset\) defines the underlying \emph{line} of our topology. 
Demand is given as a set of requests $\Rset$, where each request $r \in \Rset$ specifies an origin $o_r \in \Hset$ and a destination $d_r \in \Hset$. We assume that each request consists of one passenger.
To serve the requests, we have 
a fleet of $\sizeK$ homogeneous vehicles $\Kset$ that can carry up to capacity $\Qmax$ passengers. 
The vehicle tours are not restricted to start and end at a depot; instead, vehicles may begin and terminate their tours at any station in $\Hset$. The vehicles may skip stations that do not constitute an origin or destination of any assigned request, but they are constrained to the line structure through the so-called \emph{directionality property} introduced in~\cite{reiter_line-based_2024}: vehicles may only turn around when empty, i.e., when there are no passengers onboard.

A solution to the \lipdp consists of two parts: firstly, a \emph{tour} for each vehicle, denoting  the sequence of stations where the vehicle stops. Secondly, for each request, either the assigned vehicle and index in the tour at which the request enters the vehicle, or the decision to reject this request. This request assignment needs to respect the vehicle capacity, requiring that the number of passengers on board a vehicle does not exceed $\Qmax$ at any point in its tour. A tour may be empty, indicating that the vehicle is not used in the solution. 
An example instance is shown in \Cref{fig:example-instance}: the blue and green requests want to travel in the same direction, opposite to the orange request. In the solution shown in \Cref{fig:example-solution}, the blue and green requests are pooled together, followed by the orange request.

\begin{figure}[htb!]
    \begin{subfigure}{0.4\linewidth}
        \centering
        \includegraphics[width=\linewidth,page=1]{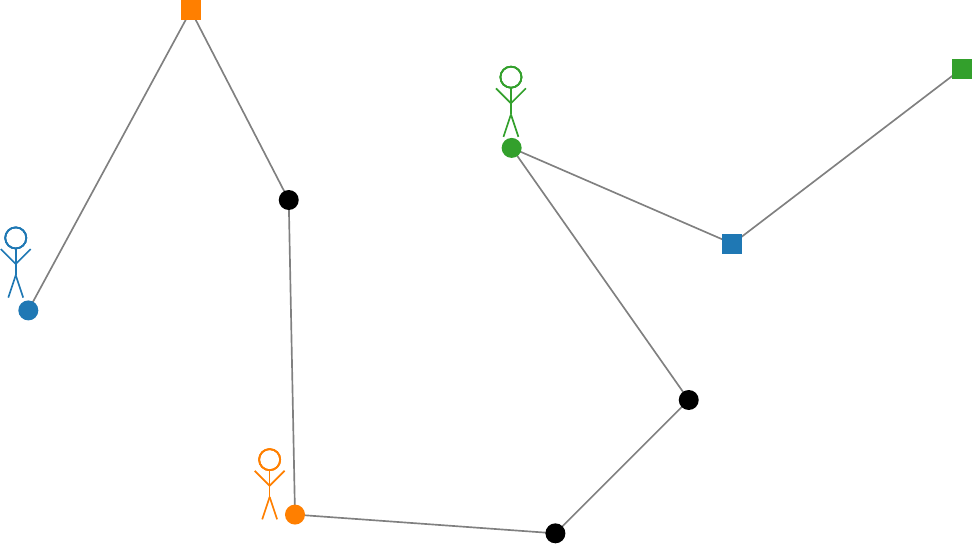}
        \caption{Line with three requests whose destinations are denoted by the squares.}
        \label{fig:example-instance}
    \end{subfigure}
    \hfill
    \begin{subfigure}{0.4\linewidth}
        \centering
        \includegraphics[width=\linewidth,page=2]{figures/00_other/example-line-requests.pdf}
        \caption{Possible solution serving all three requests. The original line is dashed.}
        \label{fig:example-solution}
    \end{subfigure}
    \caption{Example instance and solution for the \lipdp.}
    \label{fig:example-instance-and-solution}
\end{figure}

The objective function maximizes the weighted sum of the number of accepted requests and the \emph{saved distance}: the difference of the total distance traveled (by all vehicles) and direct distance between origin and destination of all accepted requests. This addresses customer satisfaction, transporting as many requests as possible whilst considering environmental concerns, minimizing the distance traveled compared to individual motorized transport.

The \lipdp can be summarized as follows:
\defoptproblem{\lipdp}
{An instance $\Iset$ with stations $\Hset$ with pairwise distances $t$, requests $\Rset$, and vehicles~$\Kset$ of capacity $\Qmax$, objective weights $\wpax$ and $\wdist$}
{A tour for each vehicle, a set of accepted requests, and an assignment of (accepted) requests to a vehicle and index in the vehicle tour that respects capacity constraints and maximizes the objective function with profits $\wpax$ per transported request and $\wdist$ per unit of saved distance.}

In this paper, we study the \emph{static} setting of the \lipdp, where all requests are known ahead of time --- in contrast to \emph{dynamic} settings, where requests appear and have to be inserted during operation.

Considering the complexity of the \lipdp, we state the following result:
\begin{thm}\label{thm:complexity:lipdp}
    The decision version of the \lipdp is \NP-complete even in the case where there is only one vehicle of capacity one and when the distances are Euclidean.
\end{thm}

\Cref{thm:complexity:lipdp} can be shown by a reduction from the \textsc{open traveling salesperson problem} which has been proven to be \NP-complete even for Euclidean instances by \cite{PAPADIMITRIOU1977237}.
The result follows from the equivalence of the \textsc{TSP} graph with the \lipdp line, by adjusting the objective function weights to ensure that all requests need to be transported. A full proof is given in the Appendix in \Cref{sec:appendix:complexity}.

\section{Modeling and Mathematical Formulation}\label{sec:model}
An overview of all notation is given in \Cref{tab:params} in the Appendix.

\subsection{Modeling Stopping Patterns}\label{sec:model:stoppingpattern}
Each request $r \in \Rset$ has an inherent direction of travel with respect to the underlying sequence of stations $\Hset$: if $o_r < d_r$, we call this \emph{ascending}, otherwise \emph{descending}.
We denote by $\Rasc$ (resp. $\Rdesc$) the set of ascending (descending) requests, and by $\dirs$ the respective direction.

Now consider a vehicle tour, i.e., the sequence of stations where a vehicle stops. Analogously to the denomination of the requests, we can divide the tour into ascending and descending \emph{sublines}. Then, two consecutive sublines in a tour are connected by a turn, and we refer to the stations where a vehicle changes direction as \emph{turn stations}. Note that there are three reasons for a subline to stop at a station $h \in \Hset$: to pick up a request, to drop off a request, or to turn.
The station in $\Hset$ where a subline starts (resp. ends) is called the \emph{start station} (\emph{end station}). 
For two consecutive sublines $s_i, s_j$, the end station of subline $s_i$ is a turn station and corresponds exactly to the start station of $s_j$. This ensures a vehicle tour is connected.
In this sense, we subdivide a tour at its turns and obtain a set of \emph{positions}, where each position is assigned a, not necessarily distinct, subline. 

Each subline defines a \emph{stopping pattern} which we encode as a vector in $\set{0,1}^n$, where a $1$ at index $h$ corresponds to stopping at station $h \in \Hset$, and a $0$ corresponds to not stopping there.
The first entry of a stopping pattern that is equal to $1$ is called the \emph{lowest} station and the last entry which is equal to $1$ is called the \emph{highest} station. 
The length $l_j$ of a stopping pattern $s_j$ (and thus of all sublines using this stopping pattern) is given by the sum of the distances between consecutive stations that are visited by the stopping pattern. 
We denote the set of all stopping patterns by $\Sset$. 
For modeling reasons, we include stopping patterns that contain only one station into this set; they have length $0$ and are called \emph{single-stop patterns}. 
The $\mathbf{0}$-vector does not model a stopping pattern, as it does not stop at any station, and is thus not included in $\Sset$. Then, there are $u := 2^\sizeH - 1$ possible stopping patterns on a sequence of $\sizeH$ stations. 
Let $\Jset := \set{1, \ldots, \sizeS}$ be the index set of $\Sset$.

Without loss of generality, each tour starts with an ascending subline, so that ascending sublines are assigned to odd positions, and descending sublines are assigned to even positions. We use $\Pasc$ and $\Pdesc$ to denote the set of ascending and descending positions, use $\Pdir$ where $\dirs$, and denote by $\Pset := \set{1, \ldots, \sizeP}$ all positions of a vehicle's tour.
Consequently, ascending requests may only be assigned to ascending positions, and analogously for descending requests.

Note that the index of a request in a vehicle's tour, which is part of the \lipdp solution, can be computed from the request's assigned position and the sublines in the vehicle tour up until this position.

\begin{figure}[hbt]
    \centering
    \includegraphics[width=0.7\linewidth,page=3]{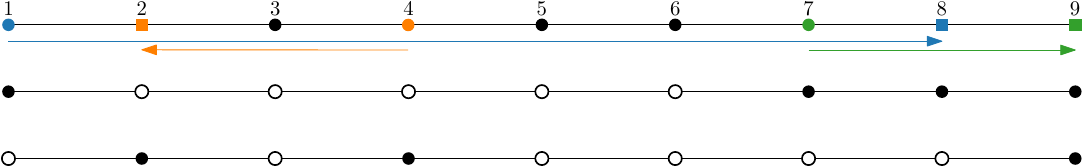}
    \caption{Line with requests (origins as colored circles, destinations as colored squares) and the two stopping patterns used in \Cref{fig:example-solution}. Filled markers are visited, unfilled markers are skipped stations.}
    \label{fig:example-patterns}
\end{figure}

Revisiting the example instance from \Cref{fig:example-instance-and-solution}, the two stopping patterns used in the solution are depicted in \Cref{fig:example-patterns}. Here, station number $1$ is the lowest station and number $9$ the highest station of the first stopping pattern, with the vehicle turning at station 9. The second stopping pattern starts with its highest station at number $9$ and its lowest station at number $2$. Note that the second stopping pattern `stops' at station $9$, as this corresponds to the end station of the previous pattern, without picking up or dropping off a request, to ensure a continuous tour.

\subsection{Mixed-Integer Linear Programming Model}\label{sec:model:milp}
The following mixed-integer linear programming model for the \lipdp uses the concept of stopping patterns by explicitly enumerating the set $\Sset$ on the given set of stations $\Hset$.

\paragraph{Positions}

For modeling purposes, we apriori fix the number of positions within the tour for each vehicle. Each position needs to be assigned exactly one stopping patterns from $\Sset$. 

\paragraph{Overlapping Requests}
Given a set of requests $\Rset$, we define a set of \emph{overlapping requests} $\Gset$ to be a nonempty subset of $\Rset$ such that, for every pair of entries $r_i, r_j \in \Gset$, 
\begin{itemize}
    \item $r_i$ and $r_j$ travel in the same direction, and
    \item  $r_i$ and $r_j$ overlap spatially, i.e., their paths along the line share at least two stations.
\end{itemize}
We denote by $\Gasc$ (resp. $\Gdesc$) the set of all such overlapping request sets in ascending (descending) direction. These sets are used to model the capacity restriction in our MILP model.

\paragraph{Parameters}
We denote by $\Left{h}{j} \in \set{0,1}^\sizeH$ (resp. $\Right{h}{j} \in \set{0,1}^\sizeH$) if station $h \in \Hset$ is the lowest (highest) station of stopping pattern $s_j \in \Sset$. Exactly one entry of $\Left{h}{j}$ ($\Right{h}{j}$) is non-zero. We use $s_j(h) \in \set{0,1}$ to denote whether stopping pattern $s_j \in \Sset$ visits station $h \in \Hset$. As before, $l_j \geq 0$ is the length of $s_j \in \Sset$.

\paragraph{Variables}
Recall that $\Jset := \set{1, \ldots, \sizeS}$ is the index set for the set of stopping patterns $\Sset$. In our model, binary variables $y_{p,k}^j$ encode if a stopping pattern $j \in \Jset$ is assigned to position $p \in \Pset$ of vehicle $k \in \Kset$. 
Binary variables $x_{p,k}^{r}$ encode whether a request $r \in \Rset$ is assigned to position $p \in \Pset$ of vehicle $k \in \Kset$, where only feasible (in terms of direction) combinations of $r$ and $p$ are allowed. Variables $\Sstart{h}{p,k}$ (resp. $\Send{h}{p,k}$) denote whether $h \in \Hset$ is the start (end) station of the stopping pattern assigned to position $p \in \Pset$ of vehicle $k \in \Kset$. Finally, continuous variables $d_k$ denote the tour length of vehicle $k \in \Kset$.

\paragraph{Model}
The full MILP model with given objective function weights $\wpax$ and $\wdist$ is as follows:
\begin{maxi!}[3]<b>
{}{\sum_{k \in \Kset} \Biggl( \sum_{\substack{\dir \in\\ \set{\asc, \desc}}} \sum_{p \in \Pdir} \sum_{r \in \Rdir} \bigl(\wpax + \wdist \cdot t_{o_r, d_r} \bigr) \, x_{p,k}^{r}  - \wdist \cdot d_k \Biggr)}{\label{eq:master-milp}}{\label{eq:master-milp:obj}}
\addConstraint{\sum_{k \in \Kset} \sum_{p \in \Pdir} x_{p,k}^{r}}{\leq 1}{\quad \forall r \in \Rdir, \dirs \label{eq:master-milp:1}}
\addConstraint{\sum_{r \in \Gset} x_{p,k}^{r}}{\leq \Qmax}{\quad \forall p \in \Pdir, k \in \Kset, \Gset \in \Gdir, \dirs \label{eq:master-milp:2}}
\addConstraint{\Sstart{h}{p,k} - \Send{h}{p-1,k}}{= 0}{\quad \forall p \in \Pset \setminus \set{1}, h \in \Hset, k \in \Kset \label{eq:master-milp:3}}
\addConstraint{\sum_{h \in \Hset} \Sstart{h}{p,k}}{= 1}{\quad \forall p \in \Pset, k \in \Kset \label{eq:master-milp:4}}
\addConstraint{\sum_{h \in \Hset} \Send{h}{p,k}}{= 1}{\quad \forall p \in \Pset, k \in \Kset \label{eq:master-milp:5}}
\addConstraint{\sum_{j \in \Jset} y_{p,k}^j}{= 1}{\quad \forall p \in \Pset, k \in \Kset \label{eq:master-milp:6}}
\addConstraint{x_{p,k}^{r} - \sum_{j \in \Jset} {y_{p,k}^j \cdot s_j(o_r) \cdot s_j(d_r)}}{\leq 0}{\quad \forall r \in \Rdir, p \in \Pdir, k \in \Kset, \dirs \label{eq:master-milp:7}}
\addConstraint{\sum_{j \in \Jset} y_{p,k}^j \cdot \Left{h}{j} - \Sstart{h}{p,k}}{\leq 0}{\quad \forall p \in \Pasc, k \in \Kset, h \in \Hset \label{eq:master-milp:8}}
\addConstraint{\sum_{j \in \Jset} y_{p,k}^j \cdot \Right{h}{j} - \Sstart{h}{p,k}}{\leq 0}{\quad \forall p \in \Pdesc, k \in \Kset, h \in \Hset \label{eq:master-milp:9}}
\addConstraint{\sum_{j \in \Jset} y_{p,k}^j \cdot \Right{h}{j} - \Send{h}{p,k}}{\leq 0}{\quad \forall p \in \Pasc, k \in \Kset, h \in \Hset \label{eq:master-milp:10}}
\addConstraint{\sum_{j \in \Jset} y_{p,k}^j \cdot \Left{h}{j} - \Send{h}{p,k}}{\leq 0}{\quad \forall p \in \Pdesc, k \in \Kset, h \in \Hset \label{eq:master-milp:11}}
\addConstraint{\sum_{p \in \Pset} \sum_{j \in \Jset} l_j \cdot y_{p,k}^j - d_k}{\leq 0}{\quad \forall k \in \Kset \label{eq:master-milp:12}}
\addConstraint{y_{p-2, k}^j + y_{p-1,k}^j - y_{p,k}^j}{\leq 1}{\quad\forall k \in \Kset, p \in \Pset\setminus \set{1, 2}, j \in \Jset : l_j = 0 \label{eq:master-milp:13}}
\addConstraint{{d_{k+1} - d_k}}{{\leq 0}}{\quad {\forall k \in \Kset \setminus \set{\sizeK}} \label{eq:master-milp:14}}
\addConstraint{d_k}{\geq 0}{\quad \forall k \in \Kset \label{eq:master-milp:19}}
\addConstraint{x_{p,k}^{r}}{\in \set{0,1}}{\quad \forall r \in \Rdir, p \in \Pdir, k \in \Kset, \dirs \label{eq:master-milp:15}}
\addConstraint{y_{p,k}^j}{\in \set{0,1}}{\quad \forall j \in \Jset, p \in \Pset, k \in \Kset \label{eq:master-milp:16}}
\addConstraint{\Sstart{h}{p,k}}{\in \set{0,1}}{\quad \forall p \in \Pset, k \in \Kset, h \in \Hset \label{eq:master-milp:17}}
\addConstraint{\Send{h}{p,k}}{\in \set{0,1}}{\quad \forall p \in \Pset, k \in \Kset, h \in \Hset \label{eq:master-milp:18}}
\end{maxi!}

Constraints \eqref{eq:master-milp:1} ensure that a request is transported at most once across all positions and all vehicles. Setting these constraints to equality would ensure that all requests are transported. Constraints \eqref{eq:master-milp:2} uphold the maximum capacity of $\Qmax$ per vehicle. Constraints \eqref{eq:master-milp:3} ensure that tours are consecutive. 
Each position in a vehicle's tour can have exactly one start and one end station, as modeled by Constraints~\eqref{eq:master-milp:4} and~\eqref{eq:master-milp:5}.
Constraints \eqref{eq:master-milp:6} ensure that each vehicle's position is assigned exactly one stopping pattern. 
Linking the $x$ and $y$ variables, Constraints~\eqref{eq:master-milp:7} ensure that, for each vehicle and each position, the assigned stopping pattern $s_j$ stops at both the origin $o_r$ and the destination $d_r$ of all requests $r$ that are assigned to the respective position of the respective vehicle.
Constraints~\eqref{eq:master-milp:8} to~\eqref{eq:master-milp:11} handle the assignment of the start and end variables, 
according to the assigned stopping pattern's lowest and highest station and the position's direction. 
Constraints \eqref{eq:master-milp:12} sum the lengths of the assigned stopping patterns to the vehicle tour length, which is minimized in the objective function.
The next constraints are \emph{symmetry-breaking} constraints:~\eqref{eq:master-milp:13} ensure that, once a tour contains two consecutive single-stop patterns (these are the only stopping patterns with $l_j=0$), the vehicle remains at this station. This ensures that stopping patterns of nonzero length are placed at the beginning of the tours.
Constraints~\eqref{eq:master-milp:14} order the lengths of tours such that vehicles of smaller index drive longer tours. Since we have homogeneous vehicles, every vehicle can drive every tour. Thus, this constraint reduces the number of solutions with the same objective function value.
Finally, Constraints~\eqref{eq:master-milp:15} to~\eqref{eq:master-milp:19} set the appropriate variable ranges.

\section{Generating Stopping Patterns}\label{sec:gen-patterns}
According to \Cref{thm:complexity:lipdp}, given a set of requests $\Rset$, selecting the most profitable stopping patterns from $\Sset$ and determining how to position these to build vehicle tours to maximize the profit based on the requests $\Rset$ is \NP-hard. 

For this reason, we now investigate a related problem to the \lipdp, where we need to find only one stopping pattern, namely the one that will gain the highest profit, subject to upholding the vehicle capacity $\Qmax$. More precisely, we assume that each request $r \in \Rset$ is associated with a \emph{reward} $\iota_r \geq 0$. As before, we have distances $t_{h,h'} \geq 0$ between stations $h,h' \in \Hset$.

\defoptproblem{\Ppattern}
{A \lipdp instance $\Iset$ with stations $\Hset$ with pairwise distances $t$, a direction $\dir$, requests $r\in \Rdir$ 
with rewards $\iota_r\ge 0$, and vehicles $\Kset$ of capacity~$\Qmax$.}
{A stopping pattern and a set of accepted requests that respects capacity constraints and maximizes the profit, i.e., the sum of rewards of accepted requests minus the length of the stopping pattern}

We also consider an uncapacitated problem version, the \Ppatternuncap problem, where the vehicles have unlimited capacity. 
Note that for a solution of the \Ppatternuncap problem, we do not need to specify the set of accepted requests explicitly, since, given a stopping pattern $s$ and a direction $\dir$, it is always best to accept all requests from $\Rdir$ whose origin and destination are contained in the stopping pattern. 

In \Cref{sec:gen-patterns:complexity}, we show that the \Ppatternuncap and \Ppattern problems are \NP-hard, and discuss exact solution approaches in \Cref{sec:gen-patterns:solving}.

\subsection{Complexity}\label{sec:gen-patterns:complexity}
In the following, we consider the decision version corresponding to the \Ppatternuncap problem: \emph{``Given an instance of the \Ppatternuncap problem and a threshold $B$, is there a stopping pattern with profit at least $B$?''}. We show that this problem is \NP-hard 
by a reduction from the \Pclique problem and conclude \NP-hardness of the \Ppattern problem.
The idea of the proof is that the stations served by a stopping pattern correspond to vertices, while the served requests correspond to edges, such that a stopping pattern with profit at least~$B$ corresponds to a clique of size at least~$b$.

\begin{thm}\label{thm:pricing-np-hard}
    The \Ppatternuncap problem is \NP-hard.
\end{thm}
\begin{proof}
    We reduce the \Pclique problem, which is one of Karp's 21 original \NP-hard problems~\citep{karp_np_complete_72}, to the \Ppatternuncap problem.

    \defdecproblem{\Pclique}
    {A graph $G$, an integer $b \geq 0$.}
    {Is there a clique of size~$b$ in $G$?}

    Let~$(G,b)$ be an instance of the \Pclique problem with $G=(V,E)$ and~$V=\set{1,\dots,n}$. To avoid confusion, we refer to the defining sets of $G$ as \emph{vertices} and \emph{edges}, while the defining sets of the graph that we will construct are referred to as \emph{nodes} and \emph{arcs}.
    Assume without loss of generality that~$n\geq b$.
    We construct an instance of the \Ppatternuncap problem as follows:
    
    Before we can add the nodes and requests corresponding to the vertices and edges of~$G$, which we call \emph{vertex nodes} and \emph{edge requests}, we define multiple auxiliary nodes and requests used in the construction of our instance. Let $\alpha, \beta, \gamma, \delta, \varepsilon$ be positive constants such that $\alpha \gg \beta \gg \gamma \gg \delta \gg \varepsilon \gg 0$.
    
    We begin by defining a sequence of nodes $(h_\mathrm{L}^{1}, \ldots, h_\mathrm{L}^{b+1}, h_\mathrm{R}^{1}, \ldots, h_\mathrm{R}^{b+1})$ consisting of $2(b+1)$ \emph{base nodes}, which we divide into a \emph{left} side and \emph{right} side, as denoted by the indices. The distance between any two consecutive base nodes is set to $2\beta$. 
    For all pairs of nodes for which the distance is not explicitly specified, we set it to the shortest-path distance with respect to the network of arcs (with specified distances).
    Between each pair of consecutive base nodes, we insert a \emph{base request} with profit~$\alpha$. As~$\alpha \gg \beta$, the optimal stopping pattern for 
    the instance defined by the components introduced so far is to visit all base nodes, yielding a profit of~$(2b+1)\alpha - (2b+1)2\beta$.

    \begin{figure}[htbp]
        \centering
        \begin{subfigure}{.13\textwidth}
            \centering
            \includegraphics[page=17,width=\textwidth]{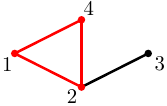}
            \caption{}
            \label{fig:clique-instance}
        \end{subfigure}
        \hspace{.05\textwidth}
        \begin{subfigure}{.8\textwidth}
            \centering
            \includegraphics[page=7,width=\textwidth]{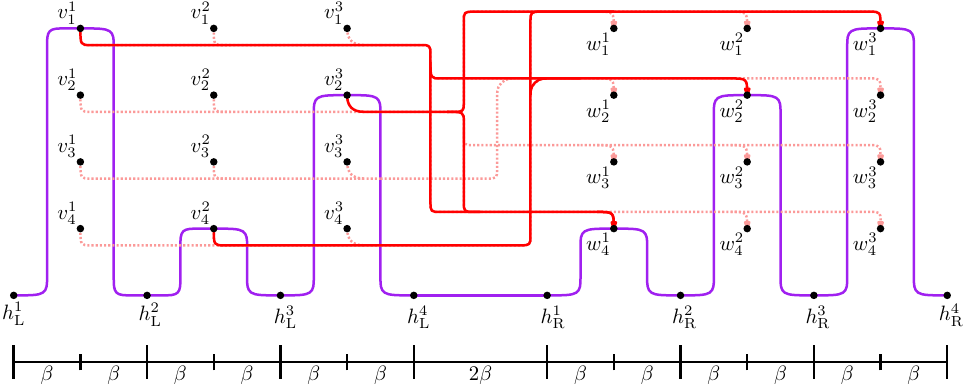}
            \caption{}
            \label{fig:clique-edge-requests}
        \end{subfigure}
        \caption{(a) An instance of \Pclique where $b=3$, with the vertices $1$, $2$, and $4$ forming a $3$-clique (highlighted in red).
        (b) The \Ppatternuncap instance constructed from the \Pclique instance in (a). An optimal stopping pattern includes the vertex nodes corresponding to the $3$-clique in (a) with six of the (red) edge requests being served by this stopping pattern. The base, particularity, and consistency requests are omitted for clarity.}
        \label{fig:pricing-np-hard-construction}
    \end{figure}

    Between each pair of consecutive base nodes on each side, i.e., between~$h_\mathrm{L}^{j}$ and~$h_\mathrm{L}^{j+1}$ (on the left) for each~$j\in[b]$, we insert~$n$ \emph{vertex nodes}~$v_{1}^{j}, \ldots, v_{n}^{j}$. Each vertex node~$v_{i}^{j}$ is at distance~$\beta$ to its neighboring base nodes~$h_\mathrm{L}^{j}$ and~$h_\mathrm{L}^{j+1}$. The pairwise distance between two vertex nodes~$v_{i}^{j}$ and~$v_{i'}^{j}$ for~$i\neq i'$ is~$2\beta$. The right side follows analogously, inserting vertex nodes~$w_{1}^{j}, \ldots, w_{n}^{j}$ between~$h_\mathrm{R}^{j}$ and~$h_\mathrm{R}^{j+1}$ for each~$j\in[b]$. Again, all distances that are not explicitly specified are computed as shortest-path distances. 
    We insert the vertex nodes in order of their index between their respective base stations. That is, between~$h_\mathrm{L}^{j}$ and~$h_\mathrm{L}^{j+1}$ we have the sequence~$(h_\mathrm{L}^{j}, v_{1}^{j}, v_{2}^{j}, \ldots, v_{n}^{j}, h_\mathrm{L}^{j+1})$. 
    See \Cref{fig:pricing-np-hard-construction} for an example of the resulting instance.

    The idea is that the vertex nodes visited by an optimal stopping pattern correspond to the vertices of a clique of size~$b$ in the original graph~$G$. However, for the instance constructed so far, the optimal stopping pattern would only have to visit the base nodes. Therefore, we insert on each side for every~$i,i'\in[n]$ and $j,j'\in[b]$ with $i\neq i'$ and~$j<j'$ a \emph{particularity request} between vertex nodes~$v_{i}^{j}$ and~$v_{i'}^{j'}$ (resp.~$w_{i}^{j}$ and~$w_{i'}^{j'}$) with profit~$\gamma$. This leads to the most profitable stopping pattern visiting exactly one vertex node~$v_{i}^{j}$ (resp.~$w_{i}^{j}$) between each pair of base nodes~$h_\mathrm{L}^{j}$ and~$h_\mathrm{L}^{j+1}$ (resp.~$h_\mathrm{R}^{j}$ and~$h_\mathrm{R}^{j+1}$). Note that visiting \emph{two} vertex nodes between the same two base nodes would incur an additional driving cost of~$2\beta$, which cannot be offset by the additional profit of the particularity requests, as~$\beta \gg \gamma$. Further, the indices~$i$ of the visited vertex nodes on each side are distinct: suppose that an optimal stopping pattern visits two vertex nodes with the same index~$i$ on the left (same for the right) side, i.e., visits vertex nodes~$v_{i}^{j}$ and~$v_{i}^{j'}$ with~$j\neq j'$. As there are no particularity requests between these two vertex nodes, the stopping pattern misses out on a profit of~$\gamma$ compared to visiting a different vertex node~$v_{i'}^{j'}$. Since~$n\geq b$, we can always visit vertex nodes with distinct indices. This results in an optimal stopping pattern having an additional profit of~$2\binom{b}{2}\gamma$ from the particularity requests.

    An optimal stopping pattern for the instance constructed so far visits~$b$ distinct vertex nodes on each side, corresponding to~$2b$ vertices in~$G$. To ensure that the vertex nodes visited on each side correspond to the same vertices, we add \emph{consistency requests} for each~$i\in[n]$ between vertex nodes~$v_{i}^{j}$ and~$w_{i}^{j'}$ for each~$j,j'\in[b]$ with profit~$\delta$. This ensures that the most profitable stopping pattern visits vertex nodes with the same indices~$i$ on both sides, as this yields an additional profit of~$b\delta$. Further, as~$\gamma \gg \delta$, we still visit vertex nodes with distinct indices on each side, as the loss of a particularity request cannot be offset by the additional profit of the consistency requests.

    To enforce that the visited vertex nodes correspond to a clique in~$G$, we lastly add for each edge~$(i,i')\in E$ an \emph{edge request} between vertex nodes~$v_{i}^{j}$ and~$w_{i'}^{j'}$ for each~$j,j'\in[b]$ with profit~$\varepsilon$. If~$G$ contains a clique of size~$b$, visiting vertex nodes corresponding to the vertices of the clique, which consists of~$\binom{b}{2}$ edges, yields a profit of~$2\binom{b}{2}\varepsilon$ from the edge requests. Conversely, if a stopping pattern visits vertex nodes corresponding to vertices that do not form a clique, it misses out on at least one edge request. See \Cref{fig:clique-edge-requests} for an example of an optimal stopping pattern and its served edge requests.

    Therefore, if and only if the original graph~$G$ contains a clique of size~$b$ we obtain a stopping pattern with profit~$(2b+1)\alpha+2\binom{b}{2}\gamma + b\delta + \binom{b}{2}\varepsilon - (2b+1)2\beta$. Otherwise, all stopping patterns have a smaller profit. Thus, setting the threshold~$B$ to~$(2b+1)\alpha+2\binom{b}{2}\gamma + b\delta + \binom{b}{2}\varepsilon - (2b+1)2\beta$ we obtain in polynomial time an instance of \Ppatternuncap that is equivalent to the \Pclique instance~$(G,b)$.
\end{proof}

Note that, by setting the capacity to be sufficiently large, i.e., $\Qmax \geq \sizeR$, we have additionally shown that the \Ppattern problem is \NP-hard.

\subsection{Solving the \Ppattern Problem}\label{sec:gen-patterns:solving}
In the following, we focus on integer programming approaches to solve the \Ppattern problem.

Due to the inherent directionality (see \Cref{sec:model:stoppingpattern}), which ensures that a stopping pattern may only serve requests traveling in the same direction, we may consider only the rewards of requests in one direction when solving the \Ppattern problem. We describe the process for the ascending direction here; the descending direction follows analogously. 

We construct an acyclic tournament digraph on stations $(1, \ldots, \sizeH)$: let $\Vset := \set{0, 1, \ldots, \sizeH, \sizeH +1}$ denote the extension by an artificial start $0$ and end depot $\sizeH + 1$. The complete set of directed arcs is given by $\set{ (g,h) \in \Vset \times \Vset : g < h}$ with arc weights $w_{g,h} \geq 0$. Each request $r \in \Rasc$ is assigned a reward $\iota_r \geq 0$.

The idea is as follows: each stopping pattern is expressed as a path between depots $0$ and $\sizeH + 1$ that contains exactly the stations visited by the stopping pattern.
The lowest station of the stopping pattern is the one which is connected to the starting depot $0$ and the highest station is the one which is connected to the end depot ${\sizeH+1}$.
The reward for a request is awarded only if both its origin and destination are visited by the stopping pattern. 
We want to find the path that maximizes the profit, i.e., the difference between collected request rewards and the sum of costs of the selected arcs.

The \Ppattern problem is a path-finding problem, which, to the best of our knowledge, has not yet been studied in combinatorial optimization. Unlike most other path-finding problems in acyclic graphs, it is \NP-hard, as shown in \Cref{thm:pricing-np-hard}.

\subsubsection{ILP for the \Ppattern Problem}\label{sec:gen-patterns:solving:ilp}
To formulate the \Ppattern problem as an ILP, we introduce binary variables $y_{g,h}$ which take the value $1$ if and only if the arc $(g,h)$ is selected, to denote which path is taken. To ensure that the capacity $\Qmax$ is respected, binary variables $x_r$ denote whether requests $r \in \Rset$ are served. The ILP model for the ascending direction is as follows:
\begin{maxi!}<b>
    {}{\sum_{r \in \Rasc} x_{r} \cdot \iota_r - \sum_{\substack{g,h \in \Vset:\\ g < h}}t_{g,h} \cdot y_{g,h}}{\label{eq:ilp-pricing-cap}}{\label{eq:ilp-pricing-cap:obj-fct}}
    \addConstraint{\sum_{\substack{h \in \Vset:\\ h \notin \set{0, n+1}}} y_{0,h}}{=1}{\label{eq:ilp-pricing-cap:entering}}
    \addConstraint{\sum_{\substack{h \in \Vset:\\ h \notin \set{0, n+1}}} y_{h, \sizeH+1}}{=1}{\label{eq:ilp-pricing-cap:leaving}}
    \addConstraint{\sum_{\substack{g \in \Vset:\\g<h}} y_{g,h} - \sum_{\substack{g \in \Vset:\\g>h}} y_{h,g}}{= 0}{\qquad \forall h \in \Vset\setminus \set{0,n+1}\label{eq:ilp-pricing-cap:flow}}
    \addConstraint{x_{r} - \sum_{\substack{i \in \Vset: \\i<o_r}}y_{i,o_r}}{\leq 0}{\qquad \forall r \in \Rasc \label{eq:ilp-pricing-cap:reward1}}
    \addConstraint{x_{r} - \sum_{\substack{i \in \Vset: \\i<d_r}}y_{i,d_r}}{\leq 0}{\qquad \forall r \in \Rasc \label{eq:ilp-pricing-cap:reward2}}
    \addConstraint{\sum_{r \in \Gset} x_r}{\leq \Qmax}{\qquad \forall \Gset \in \Gasc}{\label{eq:ilp-pricing-cap:capacity}}
    \addConstraint{y_{g,h}}{\in \set{0,1}}{\qquad \forall g,h \in \Vset: g < h}\label{eq:ilp-pricing-cap:var-y}
    \addConstraint{x_{r}}{\in \set{0,1}}{\qquad \forall r \in \Rasc}\label{eq:ilp-pricing-cap:var-x}
\end{maxi!}

The objective function~\eqref{eq:ilp-pricing-cap:obj-fct} maximizes the difference between the sum of the rewards for serviced requests and the length of the path. 
Constraint~\eqref{eq:ilp-pricing-cap:entering} ensures that exactly one arc leaving the start depot is chosen and Constraint~\eqref{eq:ilp-pricing-cap:leaving} ensures there is exactly one arc chosen that enters the end depot.
Both constraints forbid the arc $(0, \sizeH+1)$ to ensure that the corresponding stopping pattern is not the $\mathbf{0}$-vector. 
Constraints~\eqref{eq:ilp-pricing-cap:flow} are the typical flow conservation constraints.
Constraints~\eqref{eq:ilp-pricing-cap:reward1} and~\eqref{eq:ilp-pricing-cap:reward2} ensure that, for $r \in \Rasc$, ${x}_r$ is set to zero if $o_r$ and $d_r$ are not both visited by the constructed path. 
Finally, Constraints~\eqref{eq:ilp-pricing-cap:capacity} ensure that the capacity is not exceeded, similar to Constraints~\eqref{eq:master-milp:2}.

\subsubsection{MILP for the \Ppatternuncap Problem}\label{sec:gen-patterns:solving:ilp-uncap}
We additionally present a mathematical programming formulation for the related \Ppatternuncap problem. Here, when we remove the capacity restrictions~\eqref{eq:ilp-pricing-cap:capacity} from the above ILP~\eqref{eq:ilp-pricing-cap}, we may additionally relax the binary restrictions on the $x_r$ variables (i.e., we effectively obtain an MILP), while we observed that this may lead to fractional optimal solutions in the capacitated case. 
The resulting MILP formulation for the \Ppatternuncap problem in ascending direction is 
\begin{equation}\label{eq:ilp-pricing}
    \text{maximize } \eqref{eq:ilp-pricing-cap:obj-fct}\ \text{s.t. } \eqref{eq:ilp-pricing-cap:entering}-\eqref{eq:ilp-pricing-cap:reward2}, \eqref{eq:ilp-pricing-cap:var-y},\  {x_r \geq 0 \quad \forall r \in \Rasc.}
\end{equation}

\section{Branch-and-Price Algorithm for the \lipdp}\label{sec:bap}
We now return to the \lipdp and the question of exact solution methods. So far, we have seen that the \lipdp is \NP-hard, negating any hope of finding a polynomial-time solution algorithm. 
Since the set of stopping patterns $\Sset$ is exponential in size with $|\Sset| = 2^\sizeH - 1$, the MILP~\eqref{eq:master-milp} has an exponential number of variables, as it uses one variable $y^j_{p,k}$ per combination of position $p$, vehicle $k$, and stopping pattern $s_j$. 
Solving this MILP directly, explicitly creating all variables, is practically infeasible even for moderate instance sizes, see the experiments in \Cref{sec:appendix:stopping_patterns} in the Appendix.
Therefore, we apply a branch-and-price algorithm 
that allows us to start with a small number of stopping patterns, as it iteratively generates promising patterns. 

We provide an outline of our solution approach, sketched in~\Cref{fig:workflow}, before we discuss the details in the respective sections.
We will see that the tournament digraph and the optimization model that we constructed in~\Cref{sec:gen-patterns:solving} for the \Ppatternuncap problem serve perfectly to address the pricing problem in our column generation.
In the following, we will interchangeably use the index set $\Jset$ to refer to the set of stopping patterns $\Sset$.

In the context of branch-and-price, the MILP~\eqref{eq:master-milp} we want to solve is the so-called \emph{master problem} (MP). 
In a first step, we solve the LP-relaxation (LP) of the master problem. 
Due to the high number of variables, we apply column generation, which allows us to start with a small set of stopping patterns $\Jset'$ and to iteratively add profitable stopping patterns (see~\Cref{sec:bap:cg} for details). 
The used linear program for a given set of stopping patterns $\Jset'$, termed \emph{relaxed restricted master problem} (RRMP($\Jset'$)), is given in LP~\eqref{eq:rrmp} in the Appendix.
When no promising stopping patterns can be found anymore, we have obtained an optimal solution to (RRMP($\Jset$)).
If the solution is integer, it is also an optimal solution to (MP); in any case, it provides an upper bound on the objective value of (MP). 
To find a lower bound, we can compute a feasible solution by solving the restricted version of~(MP), where instead of using the full set of stopping patterns we use only the stopping patterns $\Jset'$ created during the column generation. 
In \Cref{sec:experiments:rootnodeheuristic}, we demonstrate experimentally that the solution found at this point (which corresponds to the root node of our branch-and-bound tree) is, in many cases, already close to optimal.

To find a (provably) optimal solution, the algorithm continues by branching on fractional variables, computing upper and lower bounds in the child nodes in a similar fashion as described above. This is described in more detail in \Cref{sec:bap:scheme}.

\subsection{Column Generation}\label{sec:bap:cg}
Initially, we consider only a small set of eligible stopping patterns, indexed by $\Jset'$, and solve the \emph{dual problem} (DP($\Jset'$)) of the relaxed restricted master problem (RRMP($\Jset'$)) on this reduced set.
In each step, based on the optimal solution of (DP($\Jset'$)), we augment $\Jset'$ by generating new, promising stopping patterns in the so-called \emph{pricing problem} and resolve the problem based on the new set of eligible stopping patterns.
This process is repeated until no further stopping patterns with positive reduced costs are found.
To avoid simultaneously adding variables and constraints when we include new stopping patterns, we include all single-stop patterns in the initial set $\Jset'$, so that (RRMP($\Jset'$)) always contains all constraints of type~\eqref{eq:master-milp:13}.
Note that in each iteration of the column generation, the whole set of variables $\set{y_{p,k}^{j^*}: p \in \Pset, k \in \Kset}$ corresponding to the identified stopping pattern $j^*$ is added to (RRMP($\Jset'$)). 

The linear programs for (RRMP($\Jset'$)) and (DP($\Jset'$)) are given in LPs~\eqref{eq:rrmp} and~\eqref{eq:dp} in the Appendix.

\subsubsection{Stopping Criterion}
Denote by $(z',\pi')$ a tuple of optimal solutions to the restricted problems (RRMP($\Jset'$)) and (DP($\Jset'$)), respectively. 
The solution $z'$ can be extended to a feasible solution $z''$ of (LP) with the same objective value by setting $y_{p,k}^j=0$ for all $j \in \Jset \setminus \Jset', p \in \Pset, k \in \Kset$.

Since all indices of single-stop patterns are included in the initial set $\Jset'$, the restricted dual program (DP($\Jset'$)) has the same set of variables as (DP($\Jset$)), but lacks the constraints associated with primal variables $y_{p,k}^j$ for all $j \in \Jset \setminus \Jset', p \in \Pset, k \in \Kset$. Thus, we can interpret $\pi'$ as a (not necessarily feasible) solution to (DP$(\Jset)$).
If $\pi'$ is feasible for (DP($\Jset$)), i.e., if it also fulfills the dual constraints associated with variables $y_{p,k}^j$ for all $j \in \Jset \setminus \Jset', p \in \Pset, k \in \Kset$, then $z''$ is optimal for (RRMP($\Jset$))$=$(LP).

In each step of the column generation, we verify that the index set $\Jset'$ contains the indices of stopping patterns of an optimal solution to (RRMP($\Jset$)) or identify promising stopping patterns to include in the set. To this end, we thus solve (DP($\Jset'$)), and check whether all dual constraints associated with variables $y_{p,k}^j$ for $j \in \Jset \setminus \Jset', p \in \Pset, k \in \Kset$ are satisfied.

\subsubsection{Pricing Problem}\label{sec:bap:cg:pricing-prob}
The pricing problem constitutes in identifying a promising stopping pattern or to prove that none exists which means that the index set $\Jset'$ contains the indices of stopping patterns of an optimal solution to (RRMP($\Jset$)). A straightforward approach would be to enumerate all stopping patterns with indices in $\Jset\setminus \Jset'$ and check whether the corresponding dual Constraints~\eqref{eq:dp:2} and~\eqref{eq:dp:3} are fulfilled. Any violated constraint indicates a variable $y_{p,k}^j$ with negative reduced cost. However, since $|\Jset| \in \Oh(2^n)$, this is prohibitively time-consuming and in the following we investigate alternative ways to solve the pricing problem.

\paragraph{The Reduced Cost Functions}
Denote by $\zeta_{p,k}$, $\iota_{p,k}^{r}$, $\lambda_{p,k}^{h}$, $\mu_{p,k}^{h}$, $\nu_k$, $\psi_{p,k}^j$ the dual variables associated with Constraints~\eqref{eq:master-milp:6}, \eqref{eq:master-milp:7}, \eqref{eq:master-milp:8}--\eqref{eq:master-milp:9}, \eqref{eq:master-milp:10}--\eqref{eq:master-milp:11}, \eqref{eq:master-milp:12}, and \eqref{eq:master-milp:13} in the master problem (or~\eqref{eq:rrmp:6-A}, \eqref{eq:rrmp:7-A}, \eqref{eq:rrmp:8-A}--\eqref{eq:rrmp:9-A}, \eqref{eq:rrmp:10-A}--\eqref{eq:rrmp:11-A}, \eqref{eq:rrmp:12-A}, and \eqref{eq:rrmp:13-A} in the (RRMP($\Jset$))), respectively.
Then, for any position corresponding to an ascending direction $p\in\Pasc$ and any stopping pattern $s$ which is not a single-stop pattern, the dual constraint associated to $y_{p,k}^j$ reads
\begin{equation}\tag{\ref{eq:dp:2}}
    \zeta_{p,k} - \sum_{r \in \Rasc} s_j(o_r) \cdot s_j(d_r) \cdot \iota_{p,k}^{r}  + \sum_{h \in \Hset} \left( \Left{h}{j} \cdot \lambda_{p,k}^{h} + \Right{h}{j} \cdot \mu_{p,k}^{h} \right) + l_j\nu_k\geq 0,
\end{equation}
where $s_j(h) = 1$ if $s_j$ stops at station $h \in H$, else $0$. Furthermore, $\Left{h}{j}$ (resp. $\Right{h}{j}$) is $1$ exactly when station $h \in H$ is the lowest (highest) station of stopping pattern $s_j$.
For positions corresponding to descending direction $p\in \Pdesc$ and $s_j$ with $l_j>0$ we have 
\begin{equation}\tag{\ref{eq:dp:3}}
    \zeta_{p,k} - \sum_{r \in \Rdesc} s_j(o_r) \cdot s_j(d_r) \cdot \iota_{p,k}^{r} + \sum_{h \in \Hset} \left( \Right{h}{j} \cdot \lambda_{p,k}^{h} + \Left{h}{j} \cdot \mu_{p,k}^{h} \right) + l_j\nu_k
\geq 0.
\end{equation}

In our pricing problem, we search for tuples $(j,p,k)$ for which these inequalities are not fulfilled.
For a given dual solution to (DP($\Jset'$)), specifying values for the dual variables $\zeta_{p,k}$, $\iota_{p,k}^{r}$, $\lambda_{p,k}^{h}$, $\mu_{p,k}^{h}$, $\nu_k$, and $\psi_{p,k}^j$, we thus define the \emph{reduced costs} of a stopping pattern $s_j$, in dependence of its position $p \in \Pset$ of vehicle $k \in \Kset$ as
\begin{align}
    f_{\text{rc}}^\asc(p,k) &=  - \zeta_{p,k} + \sum_{r \in \Rasc} s_j(o_r) \cdot s_j(d_r) \cdot \iota_{p,k}^{r} - \sum_{h \in \Hset} \left( \Left{h}{j} \cdot \lambda_{p,k}^{h} + \Right{h}{j} \cdot \mu_{p,k}^{h} \right) - l_j\cdot\nu_k \label{eq:red-cost-asc} \\
    f_{\text{rc}}^\desc(p,k) &= -\zeta_{p,k} + \sum_{r \in \Rdesc} s_j(o_r) \cdot s_j(d_r) \cdot \iota_{p,k}^{r} - \sum_{h \in \Hset} \left( \Right{h}{j} \cdot \lambda_{p,k}^{h} + \Left{h}{j} \cdot \mu_{p,k}^{h} \right) - l_j\cdot\nu_k \label{eq:red-cost-desc}.
\end{align}
We want to identify stopping patterns of \emph{positive} reduced costs. If $f_{\text{rc}}^\dir \leq 0$, there does not exist an improving stopping pattern; if this is the case for all $k \in \Kset$ and $p \in \Pset$, the column generation terminates.

The first term, $-\zeta_{p,k}$ is unrestricted in sign and depends only on the considered position $p$ and vehicle $k$. It can be interpreted as the opportunity costs for changing the current assignment of a stopping pattern to this position. 
The subsequent term represents a sum of rewards ($\iota_{p,k}^{r}\ge 0$ for all $r \in R, p \in \Pset, k \in \Kset$) for every request $r$ whose origin and destination stop $o_r$ and $d_r$ are visited by stopping pattern $s_j$.
The third term constitutes of a penalty (both $\lambda_{p,k}^{h}$ and $\mu_{p,k}^{h}$ are nonnegative) for choosing the lowest and highest station of the chosen stopping pattern. 
The last term $-l_j\cdot\nu_k$ with $\nu_k\ge 0$ subtracts a penalty for driven distance.
From the respective dual constraints (see~\eqref{eq:dp:11} and~\eqref{eq:dp:12} in the Appendix), we see that without symmetry-breaking constraints~\eqref{eq:master-milp:14} (or~\eqref{eq:rrmp:14-A}), $\nu_k$ could be increased up to $w_{\text{dist}}$.

\paragraph{Formulating the Pricing Problem}
The pricing problem for a given vehicle $k$ and position $p$ can thus be understood as the problem of determining one stopping pattern to be used as a subline in the direction $\dir$ of $p$, such that the reward for the requests in $\Rdir$ whose origins and destinations are visited, minus the penalty for the corresponding start and end stations and the distance driven, exceeds the opportunity costs of this combination of $(p,k)$. 
This can be interpreted as a variant of the \Ppatternuncap problem: the rewards for each request $r \in \Rdir$ are given by the dual variables $\iota_{p,k}^r$, whereas the penalties can be encoded in the arc weights of the constructed acyclic tournament digraph. We again describe the ascending case here, the descending case follows analogously. The arc weights $w_{g,h} \geq 0$ of the tournament digraph with nodes $\Vset$ and arcs $\set{(g,h) \in \Vset \times \Vset : g < h}$ are defined as
\begin{equation*}
    w_{g,h} = \begin{cases}
        \lambda_{p,k}^h &\text{ if } g = 0, h \in \Hset, \\
        \mu_{p,k}^h &\text{ if } g \in \Hset, h = \sizeH + 1, \\
        M &\text{ if } g = 0, h = n+1,\\
        t_{g,h} \cdot \nu_k &\text{ else},
    \end{cases}
\end{equation*}
with a large positive constant $M$ to forbid the arc $(0, n+1)$, ensuring the stopping pattern is feasible. 
The reward $\iota_r$ for visiting both the origin and destination of a request $r \in \Rasc$ is set to $\iota^r_{p,k}$. If the optimal profit found when solving the \Ppatternuncap problem is at least $\zeta_{p,k}$ and the length is positive, this stopping pattern is considered \emph{profitable}.

The MILP formulation of our pricing problem for a given position $p$ of vehicle $k$ is 
\begin{maxi}[2]
    {}{\sum_{r \in \Rasc} x_r \cdot \iota^r_{p,k} - \sum_{h \in \Hset} \left( \lambda_{p,k}^h \cdot y_{0,h} + \mu_{p,k}^h \cdot y_{h, \sizeH+1}\right) - \nu_k \cdot \sum_{\substack{g,h \in \Vset:\\ g < h}}t_{g,h} \cdot  y_{g,h}}{\label{eq:pricing-prob}}{}
    \addConstraint{\eqref{eq:ilp-pricing-cap:entering}-\eqref{eq:ilp-pricing-cap:reward2}, \eqref{eq:ilp-pricing-cap:var-y}}{}{}
    \addConstraint{\sum_{\substack{g,h \in \Vset:\\ 0 < g < h < \sizeH+1}} y_{g,h}}{\geq 1}{}
    \addConstraint{x_r}{\geq 0}{\quad \forall r \in \Rasc}
\end{maxi}

The additional constraint ensures that the corresponding stopping pattern is not a single-stop pattern, since at least one edge between two stations in $\Vset \setminus \set{0, \sizeH+1}$ needs to be included.

We solve the pricing problem for every $p \in \Pset, k \in \Kset$ and, at the end, add the indices of the most profitable stopping patterns to $\Jset'$ and solve (DP($\Jset'$)). 
If no profitable stopping pattern is found, the column generation terminates.

\subsection{Branch-and-Price Algorithm}\label{sec:bap:scheme}
The described column generation approach terminates with an optimal solution of (LP). However, in general, this solution is not integer. Therefore, we further develop a branch-and-price algorithm for the \lipdp, which is outlined in \Cref{fig:workflow}.

\begin{figure}[ht!]
    \centering
    \includegraphics[width=0.9\linewidth]{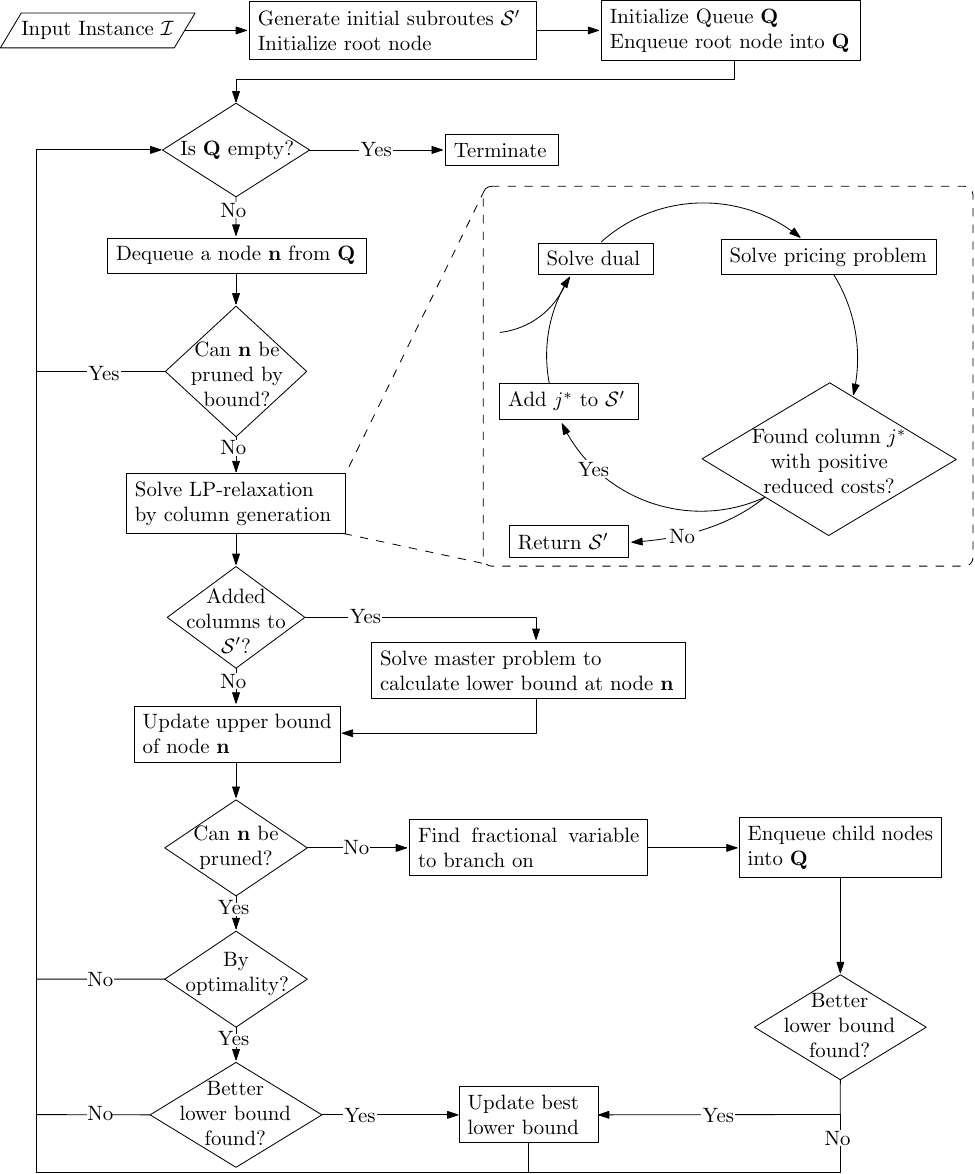}
    \caption{Flow of the Branch-and-Price Algorithm for the \lipdp.}
    \label{fig:workflow}
\end{figure}

The algorithm starts by initializing a root node with a non-empty set of initial stopping patterns $\Sset'$ that includes all single-stop patterns, as motivated in Section~\ref{sec:bap:cg}. We obtain an upper bound by solving (RRMP($\Jset'$)) and a lower bound by solving the restricted master problem with $\Jset'$, which gives us a feasible solution. The priority queue $\mathbf{Q}$, which is sorted by the nodes' parent upper bounds in non-ascending order, is initialized with the root node.

Throughout the procedure, we save the best integer solution found so far and use its objective value as a global lower bound in order to prune nodes. 

We track a single set $\Sset'$ of stopping patterns, starting with the initial ones, and extended it in each node by the profitable patterns found in our column generation.

While $\mathbf{Q}$ is non-empty, we repeatedly dequeue nodes and execute the following steps: first, we check if we can prune the selected node by bound, i.e., if its parent upper bound is less than or equal to the best lower bound. If not, we execute the column generation, extending the set of stopping patterns $\Sset'$ until no further improving pattern has been found. 
We obtain the node's upper bound by solving (RRMP($\Jset'$)) and, if new stopping patterns were added, additionally calculate an incumbent solution (lower bound) by solving (MP($\Jset'$)). 
Next, we enter the branch-and-bound phase: we check if the node can be pruned by bound, by infeasibility, or by optimality. 
If, at this point, we find a better lower bound, we update the global lower bound. Finally, if the node cannot be pruned, we branch and add the child nodes to $\mathbf{Q}$.

We consider a simple branching strategy with hierarchical branching based on the solution of (RRMP($\Jset'$)) in the current node. First, we try to branch on the $x$-variables, and, if none are fractional, we proceed with the $y$-variables. For both, we select the most fractional variable (that is, the variable with fractional part closest to $0.5$) to branch on. We create two branches, where the left branch constraints the selected variable to be equal to $0$ and the right branch constraints it to be equal to $1$. 
The resulting child nodes inherit the parent bounds and are enqueued to $\mathbf{Q}$. At the end, we return the node with the largest lower bound and the final set of stopping patterns $\mathcal{S}'$.

\section{Computational Experiments}\label{sec:experiments}
We present extensive computational experiments to evaluate our proposed branch-and-price algorithm for the \lipdp. We start by giving details on the benchmark data used (\Cref{sec:experiments:data}). In \Cref{sec:experiments:models}, we evaluate the branch-and-price algorithm in comparison to the state-of-the-art solution method for the \Plidarp. Following this, we derive a root node heuristic to quickly generate solutions (\Cref{sec:experiments:rootnodeheuristic}), for which we examine both the the impact of instance characteristics (\Cref{sec:experiments:charactistics}) and the impact of our master problem's configuration and modeling (\Cref{sec:experiments:config}).

Unless stated otherwise, the number of positions per vehicle is set to $2\,\sizeR$ and the starting set of stopping patterns $\Sset'$ consists of the stopping pattern which stops at every station in $\Hset$ and the set of all single-stop patterns. We solve the MILP~\eqref{eq:pricing-prob} as our pricing problem. The objective function weights are set to $\wpax = 10$ and $\wdist = 1$. 
All runtime experiments were repeated five times, reporting the average values here.

To compare to a state-of-the-art method, we use the aggregated location-augmented-event-based (ALAEB) MILP formulation introduced by~\cite{gaul_event-based_2022}, which we have adjusted for the line-based topology of the \lipdp by modifying the underlying event-based graph following the ideas outlined in~\cite{reiter_line-based_2024}. To practically disable the time windows in the model, all request time windows are set as large as possible (from 0 to the maximum time), and both the maximum waiting and the maximum travel time are set to the maximum time.

The code\footnote{\url{https://github.com/ReiterKM/solving-the-lidarp-by-generating-stopping-patterns}} for our implementation of the branch-and-price algorithm is written in Julia~1.11.6 and uses Gurobi~12.0.3. We ran all experiments on the high performance computing cluster ``Julia~2'' at the University of Würzburg, on a partition equipped with two AMD Epyc 7543 CPUs (64~physical cores, 128~logical processors) and 512~GB~RAM.

\subsection{Data}\label{sec:experiments:data}
All experiments are run on synthetically generated instances. We generate instances by fixing a number of bus stations and a number of requests: each request's origin and destination is sampled uniformly at random from $\Hset$, ensuring that they are distinct. For each instance size, we generate five different instances. We generate random metric distances between all stations.

The instance names encode the instance size in the format \texttt{<number of vehicles>-<number of requests>-<version>}. The version is a letter to distinguish the different instances of the same size. There is no correlation between instances ending with the same letter.

\subsection{Comparison of Branch-and-Price with the State-of-the-Art}\label{sec:experiments:models}
In our first experiment, we compare our branch-and-price algorithm to the state-of-the-art for the \Plidarp. For both approaches, we set a timeout of 3600 seconds. The seven instance sizes vary from 30 to 60 requests, in steps of five, with a single vehicle of capacity six on a line with 10 stations. 

In~\Cref{fig:model-comparison}, we plot the best incumbent and best (upper bound) found by both the ALAEB model and the branch-and-price algorithm, divided into four subplots: after 300, 900, 1800, and 3600 seconds of runtime. The same values are reported in~\Cref{tab:models:ub} in the Appendix. Note that even where values are reported, the models have not necessarily terminated with optimality.

\begin{figure}[htb!]
    \centering
    \captionsetup[subfigure]{aboveskip=-3pt,belowskip=-2pt}
    \begin{subfigure}{.48\linewidth}
        \begin{tikzpicture}
    \begin{axis}[
        table/col sep=comma,
        width=7.5cm, height=3.5cm,
        symbolic x coords={30-A, 30-B, 30-C, 30-D, 30-E, 35-A, 35-B, 35-C, 35-D, 35-E, 40-A, 40-B, 40-C, 40-D, 40-E, 45-A, 45-B, 45-C, 45-D, 45-E, 50-A, 50-B, 50-C, 50-D, 50-E, 55-A, 55-B, 55-C, 55-D, 55-E, 60-A, 60-B, 60-C, 60-D, 60-E},
        cycle list name=scatter-marks-2-plots-light-dark,
        legend style={at={(0,1.1)},anchor=west,nodes={scale=0.8, transform shape}, legend columns=4,/tikz/every even column/.append style={column sep=.2em}},
        legend cell align=left,
        xlabel={},
        ylabel={Bounds},
        xticklabels = {0, 30, 35, 40, 45, 50, 55, 60},
        xtick style={draw=none},
        xticklabel style={rotate=0,xshift=1.2em},
        ymin=-150,
        ymax=800,
        minor y tick num=1,
        clip mode=individual,
        ]
        \addplot+[only marks,on layer=m1] table[x=Instance,y=m5] {data/01_state_of_the_art/eb_bound.dat}; \label{fig:model-comparison:eb_bound}
        \addlegendentry{};
        \addplot+[only marks,on layer=m3] table[x=Instance,y=m5] {data/01_state_of_the_art/eb_incumbent.dat}; \label{fig:model-comparison:eb_incumbent}
        \addlegendentry{ALAEB};
        \addplot+[only marks,on layer=m2] table[x=Instance,y=m5] {data/01_state_of_the_art/bnp_bound.dat}; \label{fig:model-comparison:bnp_bound}
        \addlegendentry{};
        \addplot+[only marks,on layer=m4] table[x=Instance,y=m5]  {data/01_state_of_the_art/bnp_incumbent.dat}; \label{fig:model-comparison:bnp_incumbent}
        \addplot+[only marks, fill opacity = 0, draw opacity = 0] table[x=Instance,y=m60] {data/01_state_of_the_art/bnp_bound.dat};
        \addlegendentry{Branch-and-Price};
        \filldraw [fill=PKlightgray!40!white,draw opacity=0, fill opacity=0.5] (rel axis cs:0.154,0) rectangle (rel axis cs:0.293,1);
        \filldraw [fill=PKlightgray!40!white,draw opacity=0, fill opacity=0.5] (rel axis cs:0.43,0) rectangle (rel axis cs:0.57,1);
        \filldraw [fill=PKlightgray!40!white,draw opacity=0, fill opacity=0.5] (rel axis cs:0.708,0) rectangle (rel axis cs:0.848,1);
         \draw[PKlightgray]  (rel axis cs: 0,0.11) -- (rel axis cs:1,0.11);
         \node at (rel axis cs: 0, 0.05) [anchor=west] {missing};
    \end{axis}
\end{tikzpicture}
        \captionsetup{margin={1.65cm,0cm}}
        \subcaption{300 seconds}
    \end{subfigure}
    \hfill
    \begin{subfigure}{.48\linewidth}
        \begin{tikzpicture}
    \begin{axis}[
        table/col sep=comma,
        width=7.5cm, height=3.5cm,
        symbolic x coords={30-A, 30-B, 30-C, 30-D, 30-E, 35-A, 35-B, 35-C, 35-D, 35-E, 40-A, 40-B, 40-C, 40-D, 40-E, 45-A, 45-B, 45-C, 45-D, 45-E, 50-A, 50-B, 50-C, 50-D, 50-E, 55-A, 55-B, 55-C, 55-D, 55-E, 60-A, 60-B, 60-C, 60-D, 60-E},
        cycle list name=scatter-marks-2-plots-light-dark,
        xticklabel style={rotate=0,xshift=1.2em},
        xlabel={},
        ylabel={},
        xticklabels = {0, 30, 35, 40, 45, 50, 55, 60},
        xtick style={draw=none},
        yticklabels={},
        ymin=-150,
        ymax=800,
        minor y tick num=1,
        clip mode=individual
        ]
        \addplot+[only marks,on layer=m1] table[x=Instance,y=m15] {data/01_state_of_the_art/eb_bound.dat};
        \addplot+[only marks,on layer=m3] table[x=Instance,y=m15] {data/01_state_of_the_art/eb_incumbent.dat};
        \addplot+[only marks,on layer=m2] table[x=Instance,y=m15] {data/01_state_of_the_art/bnp_bound.dat};
        \addplot+[only marks,on layer=m4] table[x=Instance,y=m15] {data/01_state_of_the_art/bnp_incumbent.dat};
        \addplot+[only marks, fill opacity = 0, draw opacity = 0] table[x=Instance,y=m60] {data/01_state_of_the_art/bnp_bound.dat};
        \filldraw [fill=PKlightgray!40!white,draw opacity=0, fill opacity=0.5] (rel axis cs:0.154,0) rectangle (rel axis cs:0.293,1);
        \filldraw [fill=PKlightgray!40!white,draw opacity=0, fill opacity=0.5] (rel axis cs:0.43,0) rectangle (rel axis cs:0.57,1);
        \filldraw [fill=PKlightgray!40!white,draw opacity=0, fill opacity=0.5] (rel axis cs:0.708,0) rectangle (rel axis cs:0.848,1);
         \draw[PKlightgray]  (rel axis cs: 0,0.11) -- (rel axis cs:1,0.11);
         \node at (rel axis cs: 0, 0.05) [anchor=west] {missing};
    \end{axis}
\end{tikzpicture}
        \subcaption{900 seconds}
    \end{subfigure}
    \\
    \begin{subfigure}{.48\linewidth}
        \begin{tikzpicture}
    \begin{axis}[
        table/col sep=comma,
        width=7.5cm, height=3.5cm,
        symbolic x coords={30-A, 30-B, 30-C, 30-D, 30-E, 35-A, 35-B, 35-C, 35-D, 35-E, 40-A, 40-B, 40-C, 40-D, 40-E, 45-A, 45-B, 45-C, 45-D, 45-E, 50-A, 50-B, 50-C, 50-D, 50-E, 55-A, 55-B, 55-C, 55-D, 55-E, 60-A, 60-B, 60-C, 60-D, 60-E},
        xticklabel style={rotate=0,xshift=1.2em},
        cycle list name=scatter-marks-2-plots-light-dark,
        legend cell align=left,
        xlabel={Number of Requests},
        ylabel={Bounds},
        xticklabels = {0, 30, 35, 40, 45, 50, 55, 60},
        xtick style={draw=none},
        ymin=-150,
        ymax=800,
        minor y tick num=1,
        clip mode=individual
        ]
        \addplot+[only marks,on layer=m1] table[x=Instance,y=m30] {data/01_state_of_the_art/eb_bound.dat};
        \addplot+[only marks,on layer=m3] table[x=Instance,y=m30] {data/01_state_of_the_art/eb_incumbent.dat};
        \addplot+[only marks,on layer=m2] table[x=Instance,y=m30] {data/01_state_of_the_art/bnp_bound.dat};
        \addplot+[only marks,on layer=m4] table[x=Instance,y=m30] {data/01_state_of_the_art/bnp_incumbent.dat};
        \addplot+[only marks, fill opacity = 0, draw opacity = 0] table[x=Instance,y=m60] {data/01_state_of_the_art/bnp_bound.dat};
        \filldraw [fill=PKlightgray!40!white,draw opacity=0, fill opacity=0.5] (rel axis cs:0.154,0) rectangle (rel axis cs:0.293,1);
        \filldraw [fill=PKlightgray!40!white,draw opacity=0, fill opacity=0.5] (rel axis cs:0.43,0) rectangle (rel axis cs:0.57,1);
        \filldraw [fill=PKlightgray!40!white,draw opacity=0, fill opacity=0.5] (rel axis cs:0.708,0) rectangle (rel axis cs:0.848,1);
         \draw[PKlightgray]  (rel axis cs: 0,0.11) -- (rel axis cs:1,0.11);
         \node at (rel axis cs: 0, 0.05) [anchor=west] {missing};
    \end{axis}
\end{tikzpicture}
        \captionsetup{margin={1.7cm,0cm}}
        \subcaption{1800 seconds}
    \end{subfigure}
    \hfill
    \begin{subfigure}{.48\linewidth}
        \begin{tikzpicture}
    \begin{axis}[
        table/col sep=comma,
        width=7.5cm, height=3.5cm,
        symbolic x coords={30-A, 30-B, 30-C, 30-D, 30-E, 35-A, 35-B, 35-C, 35-D, 35-E, 40-A, 40-B, 40-C, 40-D, 40-E, 45-A, 45-B, 45-C, 45-D, 45-E, 50-A, 50-B, 50-C, 50-D, 50-E, 55-A, 55-B, 55-C, 55-D, 55-E, 60-A, 60-B, 60-C, 60-D, 60-E},
        xticklabel style={rotate=0,yshift=0em,xshift=1.2em},
        cycle list name=scatter-marks-2-plots-light-dark,
        xlabel={Number of Requests},
        ylabel={},
        yticklabels={},
        xticklabels = {, 30, 35, 40, 45, 50, 55, 60},
        xtick style={draw=none},
        ymin=-150,
        ymax=800,
        minor y tick num=1,
        clip mode=individual
        ]
        \addplot+[only marks,on layer=m1] table[x=Instance,y=m60] {data/01_state_of_the_art/eb_bound.dat};
        \addplot+[only marks,on layer=m3] table[x=Instance,y=m60] {data/01_state_of_the_art/eb_incumbent.dat};
        \addplot+[only marks,on layer=m2] table[x=Instance,y=m60] {data/01_state_of_the_art/bnp_bound.dat};
        \addplot+[only marks,on layer=m4] table[x=Instance,y=m60] {data/01_state_of_the_art/bnp_incumbent.dat};
        \filldraw [fill=PKlightgray!40!white,draw opacity=0, fill opacity=0.5] (rel axis cs:0.154,0) rectangle (rel axis cs:0.293,1);
        \filldraw [fill=PKlightgray!40!white,draw opacity=0, fill opacity=0.5] (rel axis cs:0.43,0) rectangle (rel axis cs:0.57,1);
        \filldraw [fill=PKlightgray!40!white,draw opacity=0, fill opacity=0.5] (rel axis cs:0.708,0) rectangle (rel axis cs:0.848,1);
         \draw[PKlightgray]  (rel axis cs: 0,0.11) -- (rel axis cs:1,0.11);
         \node at (rel axis cs: 0, 0.05) [anchor=west] {missing};
    \end{axis}
\end{tikzpicture}
        \subcaption{3600 seconds}
    \end{subfigure}
    \caption{Best bound (\ref*{fig:model-comparison:eb_bound}~\ref*{fig:model-comparison:bnp_bound}) 
    and best incumbent (\ref*{fig:model-comparison:eb_incumbent}~\ref*{fig:model-comparison:bnp_incumbent}) 
    found for the ALAEB model and branch-and-price algorithm, after a given number of seconds of runtime. On the bottom of each figure, the instances for which no incumbent could be found with the respective approach are indicated.
    }
    \label{fig:model-comparison}
\end{figure}
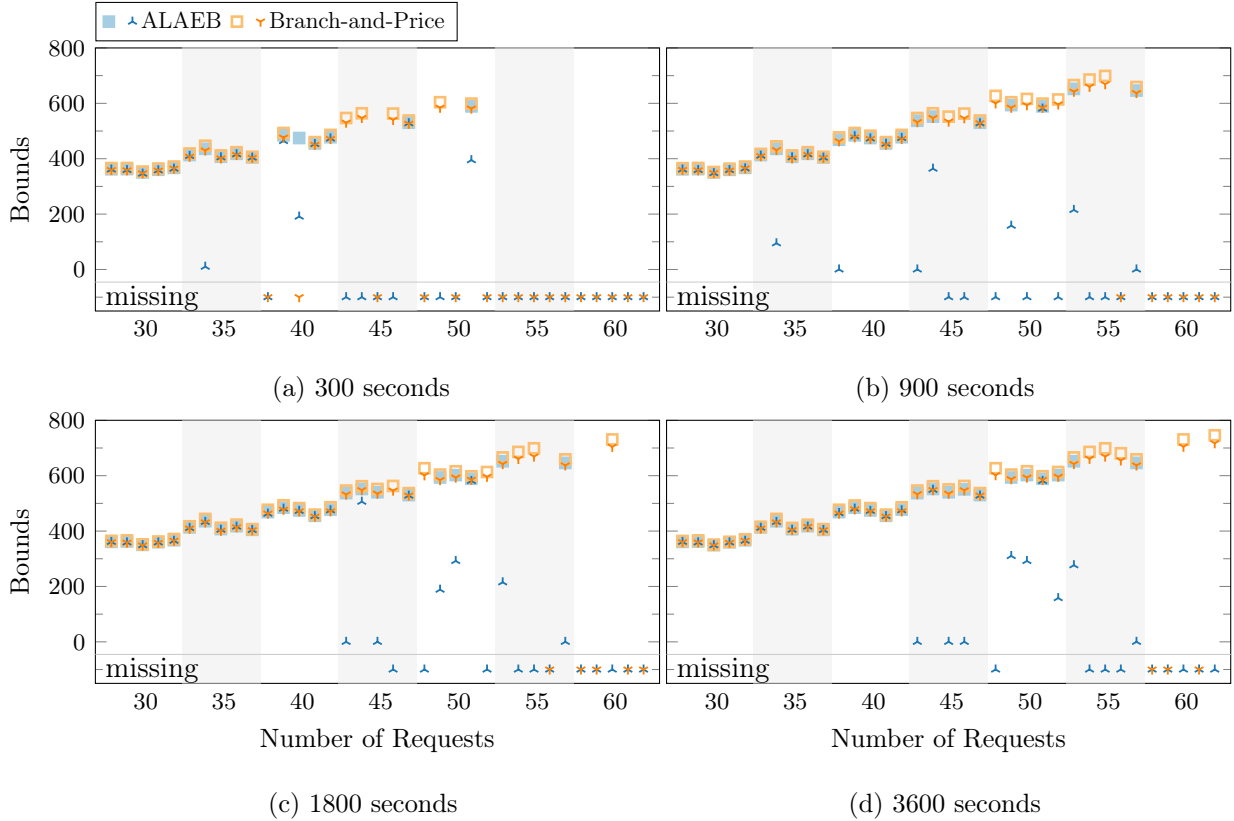

In every instance where the ALAEB model found an upper bound, the bound is lower than the one found by our branch-and-price algorithm. However, the ALAEB model struggles with finding incumbent solutions, producing a trivial solution of $0$ for four instances and no solution for nine instances after an hour of runtime. The branch-and-price algorithm found an incumbent solution for 29 instances (83\%) after already 900 seconds, and for 32 instances (91\%) after the full 3600 seconds, compared to 22 instances (63\%) and 26 instances (74\%) for the ALAEB model, respectively. 
We observe that the ALAEB model is able to quickly solve instances to optimality with up to, and including, 40 requests, whilst it could only find bounds for two instances with 55 requests, and none for instances with 60 requests, even after 3600 seconds of runtime.

Considering the MIP gap, which can be seen in the vertical gap between bound and incumbent per instance in \Cref{fig:model-comparison}, we observe that, in cases where the ALAEB model did not find the optimal solution, the MIP gap of the branch-and-price solutions is significantly smaller. The average MIP gap (of the branch-and-price solution) was $2.79\%$, with the smallest gap at $0.87\%$ for instance 1-30-C and the largest gap at $4.48\%$ for instance 1-55-C.
The virtue of the branch-and-price algorithm lies in its ability to quickly find both upper and incumbent solutions: while these may not be tight enough to prove optimality, for our instances, the MIP gap was always below $5\%$ when an incumbent was~found.

\subsubsection{Computational Time Distribution in Branch-and-Price}
Analyzing the proportionality of the computational times required for column generation and for finding an incumbent solution (i.e., solving the restricted master problem), we observe that the incumbent calculations generally dominate the runtime. In the root node, where few high-quality columns are known, the column generation accounts for an average of 30\% of the time for small instances (less than 40 requests), while the incumbent calculation requires 46\%. For larger instances, column generation averages only 4\% of runtime, compared to 93\% for the incumbent. This effect is even stronger when considering all iterations of the branch-and-price algorithm: for instances with at least 40 requests, on average, 1\% of time is spent in column generation and 98\% on calculating the incumbent. These findings are consistent with other branch-and-bound studies and indicate that, while our pricing problem is formulated as an ILP, it is not a computational bottleneck of our solution method.

The trivial solution with objective value $0$, where no vehicle drives and no request is transported, is feasible if $S'$ contains at least one single-stop pattern. There are instances where the branch-and-price algorithm did not find this trivial solution, which mans that the procedure did not solve the root node. Note, however, that for these instances, the state-of-the-art approach used for comparison did not find any feasible solution as well.

\subsection{Root Node Heuristic Method}\label{sec:experiments:rootnodeheuristic}
As observed in \Cref{sec:experiments:models}, the branch-and-price algorithm's strength lies in finding incumbent solutions much earlier than the state-of-the-art approach that we compared to, especially for instances with many requests.
Hence, we propose a root node heuristic, particularly suited for real-life applications where tight-enough bounds, and close-enough solutions, are accepted in the trade-off for a much shorter runtime.

In this heuristic, we solve only the root node with column generation and do not construct a branch-and-bound tree. We report the best bound and the best incumbent found from solving the restricted master problem on the subset of columns $\Sset'$ found in our column generation. Further, we allow at most 900 seconds of runtime, as we have found this to be sufficient following the analysis in \Cref{sec:experiments:models}.

We evaluate the heuristic under two aspects: (i) the different instance characteristics and their performance impact (\Cref{sec:experiments:charactistics}), and (ii) the configuration and modeling choices taken (\Cref{sec:experiments:config}). 
For all experiments discussed hereon, we use 15 instance sizes with 20 to 100 requests (in steps of 10), where the number of vehicles (between 1 and 5) scales with the number of requests to keep a comparable ratio of requests to available capacity. As before, we generate five different instances of the same size (and assign each a different \texttt{version}), resulting in a total of 75 instances.

\subsubsection{Instance Characteristics}\label{sec:experiments:charactistics}
We now focus on the impact of different instance characteristics on the root node heuristic. In the results presented in \Cref{sec:experiments:models}, we observe the impact of the number of requests: as these increase, the branch-and-price algorithm takes longer to find solutions. In the following experiments, we vary the number of stations of the line and the capacity of the vehicles.

To obtain a baseline, we try to compute the optimal solution to all instances using the ALAEB model with a timeout of 7200 seconds. As noted in \Cref{sec:experiments:models}, the ALAEB model is only able to solve small instances to optimality, in this case with up to 40 requests, thus limiting the amount of instances we can compare to in this section. Our root node heuristic was able to find feasible solutions to all instances, with up to 100 requests. We examine the optimality gap of the best-found incumbent from our root node heuristic to this optimal solution. In each table, we only include rows where at least one column contains an optimal objective value obtained from the ALAEB model; all other rows are omitted. 

\paragraph{The Impact of the Number of Stations}\label{sec:experiments:charactistics:stations} 
In our master problem, the number of $\Sstart{h}{p,k}$ and $\Send{h}{p,k}$ variables, which denote the lowest and highest station of each position, are dependent on the length of the line. More importantly, the number of possible stopping patterns grows exponentially in the number of stations, and the size of the pricing problem increases with the number of stations.
We evaluate instances with 10, 15, 20, and 25 stations. 
Note that instances of the same number of requests are independently generated for the different numbers of stations.

\Cref{tab:stations:incumbent} reports, for each number of stations, the incumbent solution value of our root node heuristic, the optimal solution value, and the optimality gap between the two.
The largest optimality gap of 6.25\% occurred in instance 1-20-D on 20 stations. In three instances, each with 20 requests (but different numbers of stations), the root node heuristic found the optimal solution.

As expected, the performance of the heuristic appears to decrease with the number of stations: our heuristic was able to solve 85\% of the instances with 10 stations (with an average optimality gap of -0.95\%), compared to 73\% with 15 stations (average gap -1.42\%), 69\% with 20 stations (average gap -2.06\%), and only 67\% with 25 stations (average gap -2.38\%). In case of 10 stations, the largest solved instance had 100 requests and 5 vehicles, compared to a maximum size of 80 requests and 5 vehicles for 25 stations. Up to a size of 50 requests and 3 vehicles, our heuristic solved all instances, for all number of stations.
Especially compared to the ALAEB model, which struggled with instances of 40 requests for all number of stations, this shows a greater usability of our heuristic, without sacrificing solution quality.

\begin{table}[H]
    \centering
    \caption{Incumbent objective value of the root node (RN) heuristic, optimal objective value (OV), and optimality gap for different number of stations. A dash indicates a missing value. Rows where no optimal solution was found for any configuration are omitted.
    }
    \label{tab:stations:incumbent}
    \footnotesize
    \begin{tabular}{l ccc ccc ccc ccc}
        \toprule
        Instance & \multicolumn{12}{c}{Number of Stations}\\\cmidrule(lr){2-13}
         & \multicolumn{3}{c}{10} & \multicolumn{3}{c}{15} & \multicolumn{3}{c}{20} & \multicolumn{3}{c}{25}\\\cmidrule(lr){2-4}\cmidrule(lr){5-7}\cmidrule(lr){8-10}\cmidrule(lr){11-13}
         & RN & OV & Gap & RN & OV & Gap & RN & OV & Gap & RN & OV & Gap \\
        \midrule
        1-20-A & 227 & - & - & 226 & 228 & -0.88\% & 260 & 264 & -1.52\% & 270 & 278 & -2.88\%\\
        1-20-B & 217 & 218 & -0.46\% & 225 & 229 & -1.75\% & 219 & 219 & \phantom{-}0.00\% & 241 & 247 & -2.43\%\\
        1-20-C & 237 & 241 & -1.66\% & 231 & 235 & -1.70\% & 252 & 262 & -3.82\% & 272 & 284 & -4.23\%\\
        1-20-D & 240 & 242 & -0.83\% & 244 & 246 & -0.81\% & 255 & 272 & -6.25\% & 242 & 247 & -2.02\%\\
        1-20-E & 225 & 225 & \phantom{-}0.00\% & 233 & 239 & -2.51\% & 241 & 247 & -2.43\% & 247 & 247 & \phantom{-}0.00\%\\
        \rowcolor{PKlightgray!50}1-30-A & 354 & 361 & -1.94\% & 357 & - & - & 368 & 370 & -0.54\% & 387 & 393 & -1.53\%\\
        \rowcolor{PKlightgray!50}1-30-B & 356 & - & - & 362 & 365 & -0.82\% & 383 & 391 & -2.05\% & 365 & 374 & -2.41\%\\
        \rowcolor{PKlightgray!50}1-30-C & 343 & 349 & -1.72\% & 379 & - & - & 392 & - & - & 393 & 400 & -1.75\%\\
        \rowcolor{PKlightgray!50}1-30-D & 355 & 359 & -1.11\% & 360 & - & - & 398 & - & - & 382 & 393 & -2.80\%\\
        \rowcolor{PKlightgray!50}1-30-E & 361 & - & - & 382 & - & - & 390 & 399 & -2.26\% & 384 & 396 & -3.03\%\\
        2-30-A & 348 & 349 & -0.29\% & 361 & - & - & 410 & 420 & -2.38\% & 390 & - & -\\
        2-30-B & 352 & - & - & 352 & 361 & -2.49\% & 385 & 391 & -1.53\% & 390 & 403 & -3.23\%\\
        2-30-C & 349 & 351 & -0.57\% & 377 & 382 & -1.31\% & 393 & 398 & -1.26\% & 406 & 415 & -2.17\%\\
        2-30-D & 351 & - & - & 381 & 383 & -0.52\% & 381 & - & - & 394 & 406 & -2.96\%\\
        2-30-E & 351 & - & - & 356 & - & - & 388 & 395 & -1.77\% & 386 & 394 & -2.03\%\\
        \rowcolor{PKlightgray!50}1-40-A & 462 & - & - & 514 & - & - & 535 & - & - & 526 & 534 & -1.50\%\\
        2-40-A & 481 & - & - & 477 & 484 & -1.45\% & 497 & - & - & 525 & - & -\\
        2-40-D & 447 & - & - & 501 & - & - & 519 & - & - & 511 & 530 & -3.58\%\\
        2-40-E & 472 & - & - & 511 & - & - & 518 & 523 & -0.96\% & 545 & 556 & -1.98\%\\
        \bottomrule
    \end{tabular}
\end{table}

\paragraph{The Impact of the Vehicle Capacity}\label{sec:experiments:charactistics:capacity}
We vary the vehicle capacity, keeping the number of available vehicles per instance fixed. \Cref{tab:capacity:incumbent} reports the incumbent objective value of the root node heuristic, the optimal objective value, and the optimality gap (where an optimal solution is available), for different capacities. Each instance (row) contains exactly the same requests for different capacities.

The performance of our root node heuristic improves with increasing capacity, as demonstrated by the decrease in optimality gap from -2.94\% for 3 seats to -0.20\% for 12 seats (where an optimal solution is known). For every instance, increasing the capacity decreases the optimality gap. The largest optimality gap of -6.93\% occurs for instance 1-20-D and a capacity of 3, decreasing to the smallest gap of -1.21\% for a capacity of 12 in the same instance.

\begin{table}[H]
    \centering
    \caption{Incumbent objective value of the root node (RN) heuristic, optimal objective value (OV), and optimality gap for different capacities. A dash indicates a missing value. Rows where no optimal solution was found for any configuration are omitted.
    }
    \label{tab:capacity:incumbent}
    \footnotesize
    \begin{tabular}{l ccc ccc ccc ccc}
        \toprule
        Instance & \multicolumn{12}{c}{Capacity}\\\cmidrule(lr){2-13}
         & \multicolumn{3}{c}{3} & \multicolumn{3}{c}{6} & \multicolumn{3}{c}{9} & \multicolumn{3}{c}{12}\\\cmidrule(lr){2-4}\cmidrule(lr){5-7}\cmidrule(lr){8-10}\cmidrule(lr){11-13}
         & RN & OV & Gap & RN & OV & Gap & RN & OV & Gap & RN & OV & Gap \\
        \midrule
        1-20-A & 215 & 222 & -3.15\% & 227 & - & - & 231 & - & - & 235 & 235 & \phantom{-}0.00\%\\
        1-20-B & 212 & 213 & -0.47\% & 217 & 218 & -0.46\% & 219 & 219 & \phantom{-}0.00\% & 219 & 219 & \phantom{-}0.00\%\\
        1-20-C & 228 & 233 & -2.15\% & 237 & 241 & -1.66\% & 244 & 246 & -0.81\% & 244 & 246 & -0.81\%\\
        1-20-D & 215 & 231 & -6.93\% & 240 & 242 & -0.83\% & 241 & 245 & -1.63\% & 244 & 247 & -1.21\%\\
        1-20-E & 215 & 219 & -1.83\% & 225 & 225 & \phantom{-}0.00\% & 225 & 225 & \phantom{-}0.00\% & 225 & 225 & \phantom{-}0.00\%\\
        \rowcolor{PKlightgray!50}1-30-A & 331 & 348 & -4.89\% & 354 & 361 & -1.94\% & 363 & - & - & 367 & 367 & \phantom{-}0.00\%\\
        \rowcolor{PKlightgray!50}1-30-B & 334 & 348 & -4.02\% & 356 & - & - & 362 & 365 & -0.82\% & 368 & 368 & \phantom{-}0.00\%\\
        \rowcolor{PKlightgray!50}1-30-C & 322 & 335 & -3.88\% & 343 & 349 & -1.72\% & 350 & - & - & 352 & 354 & -0.56\%\\
        \rowcolor{PKlightgray!50}1-30-D & 336 & 345 & -2.61\% & 355 & 359 & -1.11\% & 359 & 362 & -0.83\% & 365 & 365 & \phantom{-}0.00\%\\
        \rowcolor{PKlightgray!50}1-30-E & 341 & 351 & -2.85\% & 361 & - & - & 366 & - & - & 373 & 373 & \phantom{-}0.00\%\\
        2-30-A & 334 & 339 & -1.47\% & 348 & 349 & -0.29\% & 350 & 351 & -0.28\% & 353 & 353 & \phantom{-}0.00\%\\
        2-30-B & 332 & 346 & -4.05\% & 352 & - & - & 361 & - & - & 362 & - & -\\
        2-30-C & 334 & 339 & -1.47\% & 349 & 351 & -0.57\% & 350 & - & - & 354 & 355 & -0.28\%\\
        2-30-D & - & 342 & - & 351 & - & - & 357 & - & - & 360 & 360 & \phantom{-}0.00\%\\
        2-30-E & 334 & 343 & -2.62\% & 351 & - & - & 356 & - & - & 360 & 360 & \phantom{-}0.00\%\\
        \rowcolor{PKlightgray!50}1-40-A & 440 & 451 & -2.44\% & 462 & - & - & 469 & - & - & 471 & - & -\\
        \rowcolor{PKlightgray!50}1-40-C & 442 & 457 & -3.28\% & 465 & - & - & 475 & - & - & 477 & - & -\\
        \rowcolor{PKlightgray!50}1-40-D & 429 & 444 & -3.38\% & 448 & - & - & 456 & - & - & 459 & - & -\\
        \rowcolor{PKlightgray!50}1-40-E & 450 & 461 & -2.39\% & 471 & - & - & 475 & - & - & 483 & - & -\\
        2-40-C & 455 & 463 & -1.73\% & 477 & - & - & 483 & - & - & 488 & - & -\\
        2-40-D & 428 & 442 & -3.17\% & 447 & - & - & 453 & - & - & 458 & - & -\\
        \bottomrule
    \end{tabular}
\end{table}

For a capacity of 3 seats, our heuristic found a feasible solution in 75\% of instances, compared to 85\% for 6 seats, 88\% for 9 seats, and 87\% for 12 seats. Except for the case with 3 seats, the heuristic found feasible solutions for instances with up to 100 request and 5 vehicles.

These results additionally highlight that our root node heuristic is able to find solutions for much larger instances than the state-of-the-art. 
While this method remains a heuristic, the optimality gap is sufficiently small even for difficult instances to be applicable in real-life instances.

\subsubsection{Configuration}\label{sec:experiments:config}
In this section, we take a closer look at two specific decisions we have taken in modeling our master problem: firstly, we vary the number of positions per vehicle, and secondly, we consider the impact of the symmetry-breaking constraints.

For both experiments, we evaluate only the restricted master problem under varying scenarios. To precompute a promising set of stopping patterns $\Sset'$, we run the column generation in the root node for each instance until it terminates.
This results in an average of 37 patterns (min: 27, max: 49), where the number of stopping patterns seems to be independent of the size of the instance.
Then, we try to solve the restricted master problem associated with $\Sset'$ to optimality within a time limit of 900 seconds, providing a reference solution if successful.

\paragraph{The Impact of the Number of Available Positions}\label{sec:experiments:config:positions}
As shown by~\cite{lauerbach_complexity_2025}, determining the minimal number of positions in an optimal solution of the \Plidarp for a capacity $\Qmax \geq 2$ is strongly \NP-hard. Further, the authors show that the theoretical worst-case lower bound of the number of positions needed to serve all $\sizeR$ requests is $2\, \sizeR$.
Recall that the variables $x_{p,k}^r$, $y_{p,k}^j$, $\Sstart{h}{p,k}$, and $\Send{h}{p,k}$, as well as a majority of constraints in the (restricted) master problem, are indexed by $p \in \Pset$; consequently, its size depends largely on the cardinality of $\Pset$.
We want to investigate (i) how the number of available positions affects the runtime and (ii) the (smallest) number of available positions necessary to retain the objective value obtained from a reference solution as described above, depending on the number of requests. 
In the baseline case, denoted by $100\%$, we allow $2\,\sizeR$ positions per vehicle. In every other case, we reduce the amount of available positions, e.g., the $50\%$ case corresponds to $50\%$ of $2\,\sizeR$, that is, $\sizeR$ positions per vehicle.

\begin{figure}[htb!]
    \centering
    \begin{tikzpicture}
    \begin{axis}[
        table/col sep=comma,
        width=.9\linewidth, height=.3\linewidth,
        symbolic x coords={
        1-20-A,1-20-B,1-20-C,1-20-D,1-20-E,
        1-30-A,1-30-B,1-30-C,1-30-D,1-30-E,
        2-30-A,2-30-B,2-30-C,2-30-D,2-30-E,
        1-40-A,1-40-B,1-40-C,1-40-D,1-40-E,
        2-40-A,2-40-B,2-40-C,2-40-D,2-40-E,
        2-50-A,2-50-B,2-50-C,2-50-D,2-50-E,
        3-50-A,3-50-B,3-50-C,3-50-D,3-50-E,
        3-60-A,3-60-B,3-60-C,3-60-D,3-60-E,
        4-60-A,4-60-B,4-60-C,4-60-D,4-60-E,
        3-70-A,3-70-B,3-70-C,3-70-D,3-70-E,
        4-70-A,4-70-B,4-70-C,4-70-D,4-70-E},
        cycle list name=scatter-marks-6-plots,
        enlargelimits=0.03,
        clip mode=individual,
        xmin=1-20-A,
        xmax=3-60-E,
        xtick distance=5,
        xticklabels = {, 1-20, 1-30, 2-30, 1-40, 2-40, 2-50, 3-50, 3-60, },
        xticklabel style={rotate=0,yshift=0em,xshift=2em},
        xtick style={draw=none},
        ymin = 0,
        ymax=900,
        minor y tick num=1,
        xlabel={Instance},
        ylabel={Time [s]},
        legend style={at={(0,1.07)},anchor=west,nodes={scale=0.8, transform shape}, legend columns=6, /tikz/every even column/.append style={column sep=.2em}, align=left},
        ]
        \addplot+[only marks, on layer=m2] table[x=Instance,y=P10] {data/02_number_of_positions/runtime.dat};
        \addlegendentry{10\%};
        \addplot+[only marks, on layer=m2] table[x=Instance,y=P25] {data/02_number_of_positions/runtime.dat};
        \addlegendentry{25\%};
        \addplot+[only marks, on layer=m2] table[x=Instance,y=P50] {data/02_number_of_positions/runtime.dat};
        \addlegendentry{50\%};
        \addplot+[only marks, on layer=m2] table[x=Instance,y=P75] {data/02_number_of_positions/runtime.dat};
        \addlegendentry{75\%};
        \addplot+[only marks, on layer=m2] table[x=Instance,y=P100] {data/02_number_of_positions/runtime.dat};
        \addlegendentry{100\%};
        \draw[PKlightgray,dashed, on layer=m1]  (rel axis cs: 0,0.972) -- (rel axis cs:1,0.972);
        \filldraw [fill=PKlightgray!40!white,,draw opacity=0, fill opacity=0.5, on layer=m2] (rel axis cs:0.137,0) rectangle (rel axis cs:0.259,1);
        \filldraw [fill=PKlightgray!40!white,,draw opacity=0, fill opacity=0.5, on layer=m2] (rel axis cs:0.379,0) rectangle (rel axis cs:0.501,1);
        \filldraw [fill=PKlightgray!40!white,,draw opacity=0, fill opacity=0.5, on layer=m2] (rel axis cs:0.622,0) rectangle (rel axis cs:0.744,1);
        \filldraw [fill=PKlightgray!40!white,,draw opacity=0, fill opacity=0.5, on layer=m2] (rel axis cs:0.866,0) rectangle (rel axis cs:1,1);
    \end{axis}
\end{tikzpicture}
    \vspace{-2em}
    \caption{Average runtime of the restricted master problem for different number of available positions. The dashed line indicates the timeout at 900 seconds.}
    \label{fig:positions:runtime}
\end{figure}
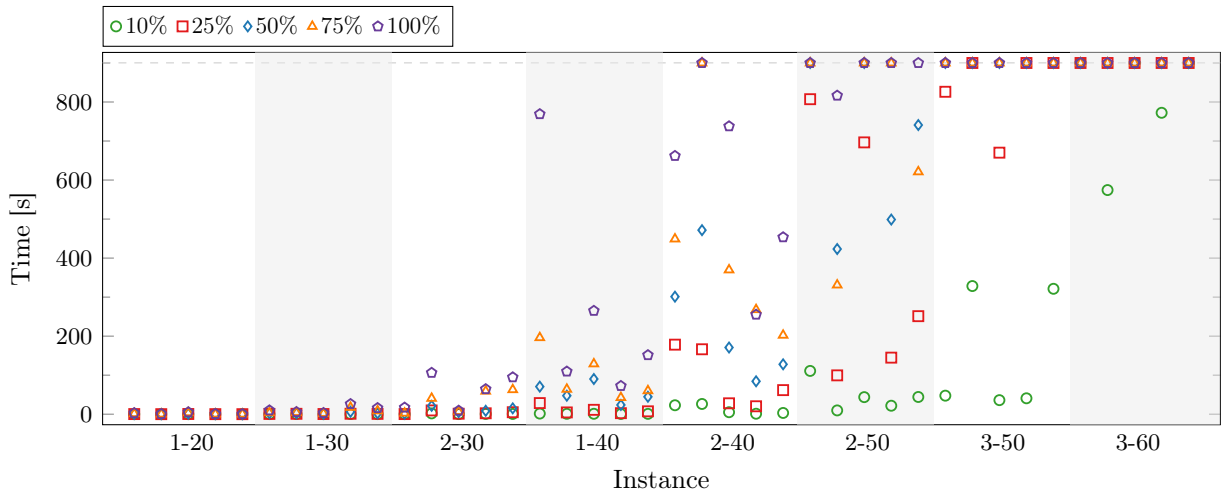

\Cref{fig:positions:runtime} shows the average runtime of the restricted master problem associated with $\Sset'$ for 10\% to 100\% of available positions. For small instances, we observe only minor deviations in runtime. For larger instances, the more positions are available, the longer the runtime. 
For the baseline case with $100\%$ of available positions, we were able to solve only one instance with 50 requests and 2 vehicles, whilst 
for the case where only $10\%$ of the $2\,\sizeR$ positions are available, we were able to solve instances with up to 60 requests. 

Next, we consider the number of positions required to preserve the objective value of the reference solution obtained with 100\% of available positions. 
Across the 62.5\% of instances for which the according restricted master problem associated with $\Sset'$ was successfully solved within the time limit, at most 12.5\% of the $2\,\sizeR$ positions are required; this occurs for two instances with 20 requests (equivalent to 5 positions).
In other words, for the restricted master problem considered here, we have not encountered instances that require a number of positions matching (or being close to) the theoretical lower bound.
On average, only 6.38\% of the positions are required, with a minimum of 2.5\% in an instance with 50 requests.

Our findings suggest that, in practice, a smaller number of available positions may be used to improve the runtime of the restricted master problem without impacting the solution quality, even though $2\,\sizeR$ available positions (and a full branch-and-price) are required to warrant optimality in general.

\paragraph{The Impact of the Symmetry Breaking Constraints}\label{sec:experiments:config:symmconstr}

In the master problem, Constraints~\eqref{eq:master-milp:13} and~\eqref{eq:master-milp:14} are symmetry-breaking constraints, introduced for modeling purposes and to reduce the search space, removing certain symmetric solutions that have the same objective function value. We evaluate their usefulness by considering four different settings: a baseline which is solving the restricted master problem as given in MILP~\eqref{eq:master-milp}, then the two scenarios without either Constraints~\eqref{eq:master-milp:13} or \eqref{eq:master-milp:14}, and lastly one setting removing both sets of symmetry-breaking constraints. 

\begin{figure}[htb!]
    \centering
    \begin{tikzpicture}
    \begin{axis}[
        table/col sep=comma,
        width=.9\linewidth, height=.3\linewidth,
        symbolic x coords={
        1-20-A,1-20-B,1-20-C,1-20-D,1-20-E,
        1-30-A,1-30-B,1-30-C,1-30-D,1-30-E,
        2-30-A,2-30-B,2-30-C,2-30-D,2-30-E,
        1-40-A,1-40-B,1-40-C,1-40-D,1-40-E,
        2-40-A,2-40-B,2-40-C,2-40-D,2-40-E,
        2-50-A,2-50-B,2-50-C,2-50-D,2-50-E},
        cycle list name=scatter-marks-6-plots,
        enlargelimits=0.03,
        clip mode=individual,
        xmin=1-20-A,
        xmax=2-50-E,
        xtick distance=5,
        xticklabels = {, 1-20,1-30,2-30,1-40,2-40,2-50},
        xticklabel style={rotate=0,yshift=0em,xshift=2.5em},
        xtick style={draw=none},
        ymin = 0,
        ymax=900,
        minor y tick num=1,
        xlabel={Instance},
        ylabel={Time [s]},
        legend style={at={(0,1.07)},anchor=west,nodes={scale=0.8, transform shape}, legend columns=6, /tikz/every even column/.append style={column sep=.2em}, align=left},
        ]
        \addplot+[only marks,on layer=m1] table[x=Instance,y=1a_baseline] {data/03_symmetry_breaking/runtime.dat};
        \addlegendentry{Baseline};
        \addplot+[only marks,on layer=m1] table[x=Instance,y=1n_no_single_stop] {data/03_symmetry_breaking/runtime.dat};
        \addlegendentry{w/o \eqref{eq:master-milp:13}};
        \addplot+[only marks,on layer=m1] table[x=Instance,y=1o_no_route_length] {data/03_symmetry_breaking/runtime.dat};
        \addlegendentry{w/o \eqref{eq:master-milp:14}};
        \addplot+[only marks,on layer=m1] table[x=Instance,y=no_symm_breaking] {data/03_symmetry_breaking/runtime.dat};
        \addlegendentry{w/o \eqref{eq:master-milp:13} and  \eqref{eq:master-milp:14}};
        \draw[PKlightgray,dashed]  (rel axis cs: 0,0.972) -- (rel axis cs:1,0.972);
        \filldraw [fill=PKlightgray!40!white,draw opacity=0, fill opacity=0.5, on layer=m2] (rel axis cs:0.175,0) rectangle (rel axis cs:0.338,1);
        \filldraw [fill=PKlightgray!40!white,draw opacity=0, fill opacity=0.5, on layer=m2] (rel axis cs:0.501,0) rectangle (rel axis cs:0.664,1);
        \filldraw [fill=PKlightgray!40!white,draw opacity=0, fill opacity=0.5, on layer=m2] (rel axis cs:0.827,0) rectangle (rel axis cs:1,1);
    \end{axis}
\end{tikzpicture}
    \vspace{-.5em}
    \caption{Runtime of the restricted master problem when removing specific (symmetry-breaking) constraints. The dashed line indicates the timeout at 900 seconds.}
    \label{fig:symmconstr:runtime}
\end{figure}

\Cref{fig:symmconstr:runtime} shows the runtime of the restricted master problem for the four different settings. These values can also be found in~\Cref{tab:symmconstr:runtime} in the Appendix.

The baseline setting is able to solve the most, and the largest, instances with up to 50 requests.
The Constraints~\eqref{eq:master-milp:13} are clearly advantageous, as the setting without these constraints is only able to solve 12 instances with up to 30 requests. Further, the runtime in the setting without Constraints~\eqref{eq:master-milp:13} is repeatedly higher than in the baseline. As for Constraints~\eqref{eq:master-milp:14}, while the runtime fluctuates and is sometimes lower without these constraints than in the baseline setting, the baseline was able to solve two more instances.

This experiment verifies that the chosen modeling restrictions are beneficial to our restricted master problem, effectively reducing the size of the search space. 

\section{Conclusion}\label{sec:conclusion}
This work introduces the \lipdp as a semi-flexible on-demand transportation system, combining the underlying spatial structure of a bus line with temporal flexibility,  and it
proposes an MILP formulation, an exact branch-and-price algorithm, as well as a heuristic variant of the approach to solve the problem.
Before detailing our solution approach for the \lipdp, we study the related \Ppattern problem, where only a single profitable stopping pattern is to be generated. We show that this problem and even its uncapacitated version, the \Ppatternuncap problem, are \NP-hard, and propose MILP formulations to solve both problems. 
Our solution approaches for the \lipdp are based on combining stopping patterns to vehicle tours, which allows an iterative generation of stopping patterns by repeatedly solving the \Ppatternuncap problem.

Computational results show that the branch-and-price algorithm is strong in finding feasible integer solutions and upper bounds early on, solving instances with up to 55 requests in 900 seconds, within a MIP gap of $5\%$. Our proposed root node heuristic is able to solve very large instances, with up to 100 requests, within its time limit of 900 seconds, consistently reaching average optimality gaps below 5\% across various settings and instance parameters.
Comparison with a state-of-the-art approach reveals that on 
small instances, the state-of-the-art approach outperforms our solution approach when it comes to solving the problem to optimality. However, on large instances, where the state-of-the-art approach struggles to find even a feasible solution in a reasonable time, our approach is able to find good solutions comparatively fast.

When it comes to extensions of our work and further experiments, it would be interesting to see whether our approach remains competitive when (very large or soft) time window constraints are introduced.
We believe that demand for semi-flexible passenger transport is likely to increase when these systems can be used spontaneously, i.e., without booking several hours in advance, which would allow for better pooling rates. For the operation of such a system, fast solution approaches to dynamic problem variants are needed. Future work may therefore focus on adapting the proposed approach for the dynamic context. 
The problem to generate stopping patterns is also relevant in the context of schedule-based transport, when designing express lines, or, as an ad-hoc control measure to recover to an existing timetable in case of delays. It would be interesting to see whether our ideas and solution approaches could contribute towards solving these problems.

\paragraph{Acknowledgement} The HPC-Cluster ``Julia 2'' was funded as DFG project as ``Forschungsgroßgerät nach Art 91b GG'' under INST 93/1145-1 FUGG. 

\theendnotes
\bibliography{references}

@inproceedings{reiter_line-based_2024,
	address = {Dagstuhl, Germany},
	series = {Open {Access} {Series} in {Informatics} ({OASIcs})},
	title = {The {Line}-{Based} {Dial}-a-{Ride} {Problem}},
	volume = {123},
	isbn = {978-3-95977-350-8},
	doi = {10.4230/OASIcs.ATMOS.2024.14},
	booktitle = {24th {Symposium} on {Algorithmic} {Approaches} for {Transportation} {Modelling}, {Optimization}, and {Systems} ({ATMOS} 2024)},
	publisher = {Schloss Dagstuhl – Leibniz-Zentrum für Informatik},
	author = {Reiter, Kendra and Schmidt, Marie and Stiglmayr, Michael},
	editor = {Bouman, Paul C. and Kontogiannis, Spyros C.},
	year = {2024},
	note = {ISSN: 2190-6807},
	pages = {14:1--14:20},
}

@Article{gaul25tight,
  author        = {Gaul, Daniela and Klamroth, Kathrin and Pfeiffer, Christian and Stiglmayr, Michael and Schulz, Arne},
  journal       = {European Journal of Operational Research},
  title         = {A Tight Formulation for the Dial-a-Ride Problem},
  year          = {2025},
  issn          = {0377-2217},
  doi           = {10.1016/j.ejor.2024.09.028},
  publisher     = {Elsevier BV},
}

@inproceedings{lauerbach_complexity_2025,
	address = {Cham},
	title = {The {Complexity} of {Counting} {Turns} in the {Line}-{Based} {Dial}-a-{Ride} {Problem}},
	isbn = {978-3-031-82697-9},
	doi = {10.1007/978-3-031-82697-9_7},
	language = {en},
	booktitle = {{SOFSEM} 2025: {Theory} and {Practice} of {Computer} {Science}},
	publisher = {Springer Nature Switzerland},
	author = {Lauerbach, Antonio and Reiter, Kendra and Schmidt, Marie},
	editor = {Královič, Rastislav and Kůrková, Věra},
	year = {2025},
	pages = {85--98},
}

@article{gaul_event-based_2022,
	title = {Event-based {MILP} models for ridepooling applications},
	volume = {301},
	issn = {0377-2217},
	doi = {10.1016/j.ejor.2021.11.053},
	number = {3},
	journal = {European Journal of Operational Research},
	author = {Gaul, Daniela and Klamroth, Kathrin and Stiglmayr, Michael},
	month = sep,
	year = {2022},
	pages = {1048--1063},
}

@inproceedings{gaul_solving_2021,
	title = {Solving the {Dynamic} {Dial}-a-{Ride} {Problem} {Using} a {Rolling}-{Horizon} {Event}-{Based} {Graph}},
	copyright = {https://creativecommons.org/licenses/by/4.0/legalcode},
	doi = {10.4230/OASIcs.ATMOS.2021.8},
	language = {en},
	booktitle = {21st {Symposium} on {Algorithmic} {Approaches} for {Transportation} {Modelling}, {Optimization}, and {Systems} ({ATMOS} 2021)},
	publisher = {Schloss-Dagstuhl - Leibniz Zentrum für Informatik},
	author = {Gaul, Daniela and Klamroth, Kathrin and Stiglmayr, Michael},
	year = {2021},
}

@InProceedings{barth_line-based_2025,
  author =	{Barth, Jonas and Reiter, Kendra and Schmidt, Marie},
  title =	{{The Line-Based Dial-a-Ride Problem with Transfers}},
  booktitle =	{25th Symposium on Algorithmic Approaches for Transportation Modelling, Optimization, and Systems (ATMOS 2025)},
  pages =	{17:1--17:20},
  series =	{Open Access Series in Informatics (OASIcs)},
  ISBN =	{978-3-95977-404-8},
  ISSN =	{2190-6807},
  year =	{2025},
  volume =	{137},
  editor =	{Sauer, Jonas and Schmidt, Marie},
  publisher =	{Schloss Dagstuhl -- Leibniz-Zentrum f{\"u}r Informatik},
  address =	{Dagstuhl, Germany},
  URN =		{urn:nbn:de:0030-drops-247736},
  doi =		{10.4230/OASIcs.ATMOS.2025.17}
}

@article{PAPADIMITRIOU1977237,
title = {The {Euclidean} travelling salesman problem is {NP}-complete},
journal = {Theoretical Computer Science},
volume = {4},
number = {3},
pages = {237-244},
year = {1977},
issn = {0304-3975},
doi = {10.1016/0304-3975(77)90012-3},
author = {Christos H. Papadimitriou},
}

@article{ruland_pickup_1997,
	title = {The pickup and delivery problem: {Faces} and branch-and-cut algorithm},
	volume = {33},
	issn = {0898-1221},
	shorttitle = {The pickup and delivery problem},
	doi = {10.1016/S0898-1221(97)00090-4},
	number = {12},
	journal = {Computers \& Mathematics with Applications},
	author = {Ruland, K. S. and Rodin, E. Y.},
	month = jun,
	year = {1997},
	pages = {1--13},
}

@article{kalantari_algorithm_1985,
	title = {An algorithm for the traveling salesman problem with pickup and delivery customers},
	volume = {22},
	issn = {0377-2217},
	doi = {10.1016/0377-2217(85)90257-7},
	number = {3},
	journal = {European Journal of Operational Research},
	author = {Kalantari, Bahman and Hill, Arthur V. and Arora, Sant R.},
	month = dec,
	year = {1985},
	pages = {377--386},
}

@article{lysgaard_pyramidal_2010,
	title = {The pyramidal capacitated vehicle routing problem},
	volume = {205},
	issn = {0377-2217},
	doi = {10.1016/j.ejor.2009.11.029},
	number = {1},
	journal = {European Journal of Operational Research},
	author = {Lysgaard, Jens},
	month = aug,
	year = {2010},
	pages = {59--64},
}

@techreport{vdv_linien-_2025,
    author = {{Verband Deutscher Verkehrsunternehmen e.V.}},
	type = {{VDV}-{Leitfaden}},
	title = {Linien- und {Bedarfsverkehre} in der {Region}: {Integriert}, datenbasiert, effizient},
	url = {https://www.vdv.de/vdv-mitteilung-10017-linien-und-bedarfsverkehr-leitfaden-on-demand-05-2025.pdfx},
    editor = {Deutsch, Volker},
	number = {10017},
	institution = {Verband Deutscher Verkehrsunternehmen e.V. (VDV)},
	year = {2025},
    month = may,
    series = {VDV-Mitteilung},
}

@techreport{agora_verkehrswende_mobilitatsoffensive_2023,
    author = {Reuter, Christian and Fritz, Christian and Lannefeld, Marvin and Ritschny, Jakub},
	title = {Mobilitätsoffensive für das {Land}. {Wie} {Kommunen} mit flexiblen {Kleinbussen} den
{ÖPNV} von morgen gestalten können.},
	language = {de},
	institution = {Agora Verkehrswende},
	year = {2023},
    number = {92-2023-DE},
    url = {https://www.agora-verkehrswende.de/veroeffentlichungen/mobilitaetsoffensive-fuer-das-land/},
}

@article{cordeau_branch-and-cut_2006,
	title = {A {Branch}-and-{Cut} {Algorithm} for the {Dial}-a-{Ride} {Problem}},
	volume = {54},
	issn = {0030-364X},
	doi = {10.1287/opre.1060.0283},
	number = {3},
	journal = {Operations Research},
	author = {Cordeau, Jean-François},
	month = jun,
	year = {2006},
    publisher = {INFORMS},
	pages = {573--586},
}

@article{ropke_models_2007,
	title = {Models and branch-and-cut algorithms for pickup and delivery problems with time windows},
	volume = {49},
	issn = {1097-0037},
	doi = {10.1002/net.20177},
	language = {en},
	number = {4},
	journal = {Networks},
	author = {Ropke, Stefan and Cordeau, Jean-François and Laporte, Gilbert},
	year = {2007},
	pages = {258--272},
}

@article{mehlert1998angebotsbezeichnungen,
  title={Angebotsbezeichnungen bei alternativen {Bedienungsformen}},
  author={Mehlert, Christian},
  journal={Nahverkehr},
  volume={16},
  pages={56--58},
  year={1998},
  publisher={ALBA FACHVERLAG}
}

@article{archetti_capacitated_2009,
	title = {The capacitated team orienteering and profitable tour problems},
	volume = {60},
	issn = {0160-5682},
	doi = {10.1057/palgrave.jors.2602603},
	number = {6},
	journal = {Journal of the Operational Research Society},
	author = {Archetti, Claudia and Feillet, Dominique and Hertz, Alain and Speranza, Maria Grazia},
	month = jun,
	year = {2009},
	note = {Publisher: Taylor \& Francis},
	pages = {831--842},
}

@article{archetti_optimal_2013,
	series = {Seventh {International} {Conference} on {Graphs} and {Optimization} 2010},
	title = {Optimal solutions for routing problems with profits},
	volume = {161},
	issn = {0166-218X},
	doi = {10.1016/j.dam.2011.12.021},
	number = {4},
	journal = {Discrete Applied Mathematics},
	author = {Archetti, Claudia and Bianchessi, Nicola and Speranza, Maria Grazia},
	month = mar,
	year = {2013},
	pages = {547--557},
}

@book{siefer_handbuch_2023,
	address = {Bonn},
	title = {Handbuch zur {Planung} flexibler {Bedienungsformen} im Ö{PNV}: ein {Beitrag} zur {Sicherung} der {Daseinsvorsorge} in nachfrageschwachen {Räumen}},
	isbn = {978-3-87994-551-1},
	shorttitle = {Handbuch zur {Planung} flexibler {Bedienungsformen} im Ö{PNV}},
	language = {de},
	publisher = {Bundesinstitut für Bau-, Stadt- und Raumforschung im Bundesamt für Bauwesen und Raumordnung},
	author = {Siefer, Thomas and Sievers, Nina and Heemsoth, Jan Peter},
	collaborator = {Kistner, Rafael and Biermanski, Lucas and {Technische Universität Braunschweig} and {Bundesinstitut für Bau-, Stadt- und Raumforschung} and {Deutschland}},
    month = oct,
	year = {2023},
}

@inproceedings{karp_np_complete_72,
  author       = {Richard M. Karp},
  editor       = {Raymond E. Miller and
                  James W. Thatcher},
  title        = {Reducibility Among Combinatorial Problems},
  booktitle    = {Complexity of Computer Computations},
  series       = {The {IBM} Research Symposia Series},
  pages        = {85--103},
  publisher    = {Plenum Press, New York},
  year         = {1972},
  doi          = {10.1007/978-1-4684-2001-2_9},
}

@article{errico_survey_2013,
	title = {A survey on planning semi-flexible transit systems: {Methodological} issues and a unifying framework},
	volume = {36},
	issn = {0968-090X},
	shorttitle = {A survey on planning semi-flexible transit systems},
	doi = {10.1016/j.trc.2013.08.010},
	language = {en},
	journal = {Transportation Research Part C: Emerging Technologies},
	author = {Errico, Fausto and Crainic, Teodor Gabriel and Malucelli, Federico and Nonato, Maddalena},
	month = nov,
	year = {2013},
	pages = {324--338},
}

@article{parragh_survey_2008,
	title = {A survey on pickup and delivery problems},
	volume = {58},
	issn = {1614-631X},
	doi = {10.1007/s11301-008-0036-4},
	language = {en},
	number = {2},
	journal = {Journal für Betriebswirtschaft},
	author = {Parragh, Sophie N. and Doerner, Karl F. and Hartl, Richard F.},
	month = jun,
	year = {2008},
	pages = {81--117},
}

@article{berbeglia_static_2007,
	title = {Static pickup and delivery problems: a classification scheme and survey},
	volume = {15},
	issn = {1863-8279},
	shorttitle = {Static pickup and delivery problems},
	doi = {10.1007/s11750-007-0009-0},
	language = {en},
	number = {1},
	journal = {TOP},
	author = {Berbeglia, Gerardo and Cordeau, Jean-François and Gribkovskaia, Irina and Laporte, Gilbert},
	month = jul,
	year = {2007},
	pages = {1--31},
}

@inproceedings{mallach_refined_2025,
	address = {Dagstuhl, Germany},
	series = {Open {Access} {Series} in {Informatics} ({OASIcs})},
	title = {Refined {Integer} {Programs} and {Polyhedral} {Results} for the {Target} {Visitation} {Problem}},
	volume = {137},
	isbn = {978-3-95977-404-8},
	doi = {10.4230/OASIcs.ATMOS.2025.8},
	booktitle = {25th {Symposium} on {Algorithmic} {Approaches} for {Transportation} {Modelling}, {Optimization}, and {Systems} ({ATMOS} 2025)},
	publisher = {Schloss Dagstuhl – Leibniz-Zentrum für Informatik},
	author = {Mallach, Sven},
	editor = {Sauer, Jonas and Schmidt, Marie},
	year = {2025},
	pages = {8:1--8:17},
}

@article{schobel_eigenmodel_2017,
	title = {An eigenmodel for iterative line planning, timetabling and vehicle scheduling in public transportation},
	volume = {74},
	issn = {0968-090X},
	doi = {10.1016/j.trc.2016.11.018},
	journal = {Transportation Research Part C: Emerging Technologies},
	author = {Schöbel, Anita},
	month = jan,
	year = {2017},
	pages = {348--365},
}

@inproceedings{roth_energy-efficient_2025,
	address = {Dagstuhl, Germany},
	series = {Open {Access} {Series} in {Informatics} ({OASIcs})},
	title = {Energy-{Efficient} {Line} {Planning} by {Implementing} {Express} {Lines}},
	volume = {137},
	isbn = {978-3-95977-404-8},
	doi = {10.4230/OASIcs.ATMOS.2025.18},
	booktitle = {25th {Symposium} on {Algorithmic} {Approaches} for {Transportation} {Modelling}, {Optimization}, and {Systems} ({ATMOS} 2025)},
	publisher = {Schloss Dagstuhl – Leibniz-Zentrum für Informatik},
	author = {Roth, Sarah and Schöbel, Anita},
	editor = {Sauer, Jonas and Schmidt, Marie},
	year = {2025},
	note = {ISSN: 2190-6807},
	pages = {18:1--18:21},
}

@article{borndorfer_column-generation_2007,
	title = {A {Column}-{Generation} {Approach} to {Line} {Planning} in {Public} {Transport}},
	volume = {41},
	issn = {0041-1655},
	doi = {10.1287/trsc.1060.0161},
	number = {1},
	journal = {Transportation Science},
	author = {Borndörfer, Ralf and Grötschel, Martin and Pfetsch, Marc E.},
	month = feb,
	year = {2007},
	note = {Publisher: INFORMS},
	pages = {123--132},
}

@article{gatt_solving_2025,
	title = {Solving the line planning problem with service-levels using a column generation-based heuristic algorithm},
	volume = {14},
	issn = {2192-4376},
	doi = {10.1016/j.ejtl.2025.100164},
	journal = {EURO Journal on Transportation and Logistics},
	author = {Gatt, Hector and Freche, Jean-Marie and Lehuédé, Fabien and Yeung, Thomas G.},
	month = jan,
	year = {2025},
	pages = {100164},
}

@article{cacchiani_column_2008,
	title = {A column generation approach to train timetabling on a corridor},
	volume = {6},
	issn = {1614-2411},
	doi = {10.1007/s10288-007-0037-5},
	language = {en},
	number = {2},
	journal = {4OR},
	author = {Cacchiani, Valentina and Caprara, Alberto and Toth, Paolo},
	month = jun,
	year = {2008},
	pages = {125--142},
}

@article{gkiotsalitis_subline_2022,
	title = {Subline frequency setting for autonomous minibusses under demand uncertainty},
	volume = {135},
	issn = {0968-090X},
	doi = {10.1016/j.trc.2021.103492},
	language = {en},
	journal = {Transportation Research Part C: Emerging Technologies},
	author = {Gkiotsalitis, Konstantinos and Schmidt, Marie and van der Hurk, Evelien},
	month = feb,
	year = {2022},
	pages = {103492},
}

@article{liu_optimizing_2023,
	title = {Optimizing public transport transfers by integrating timetable coordination and vehicle scheduling},
	volume = {184},
	issn = {0360-8352},
	doi = {10.1016/j.cie.2023.109577},
	journal = {Computers \& Industrial Engineering},
	author = {Liu, Tao and Ji, Wen and Gkiotsalitis, Konstantinos and Cats, Oded},
	year = {2023},
	pages = {109577},
}

@article{HERNANDEZPEREZ2009,
author = {Hipólito Hernández-Pérez and Juan-José Salazar-González},
title = {The multi-commodity one-to-one pickup-and-delivery traveling salesman problem},
journal = {European Journal of Operational Research},
volume = {196},
number = {3},
pages = {987--995},
year = {2009},
issn = {0377-2217},
doi = {10.1016/j.ejor.2008.05.009}
}

@article{LETCHFORDSG2016,
author = {Adam N. Letchford and Juan-José Salazar-González},
title = {Stronger multi-commodity flow formulations of the (capacitated) sequential ordering problem},
journal = {European Journal of Operational Research},
volume = {251},
number = {1},
pages = {74--84},
year = {2016},
issn = {0377-2217},
doi = {10.1016/j.ejor.2015.11.001}
}

@ARTICLE{Ascheuer2000,
	author = {Ascheuer, Norbert and Jünger, Michael and Reinelt, Gerhard},
	title = {A Branch \& cut algorithm for the asymmetric traveling salesman problem with precedence constraints},
	year = {2000},
	journal = {Computational Optimization and Applications},
	volume = {17},
	number = {1},
	pages = {61--84},
	doi = {10.1023/A:1008779125567}
}

@InProceedings{MaxRPSP2014,
author="Beerenwinkel, Niko
and Beretta, Stefano
and Bonizzoni, Paola
and Dondi, Riccardo
and Pirola, Yuri",
editor="Dediu, Adrian-Horia
and Mart{\'i}n-Vide, Carlos
and Sierra-Rodr{\'i}guez, Jos{\'e}-Luis
and Truthe, Bianca",
title="Covering Pairs in Directed Acyclic Graphs",
booktitle="Language and Automata Theory and Applications",
year="2014",
publisher="Springer International Publishing",
address="Cham",
pages="126--137",
isbn="978-3-319-04921-2",
    doi = {10.1007/978-3-319-04921-2_10}
}

@article{Hildenbrandt2019,
author = {Achim Hildenbrandt},
title = {A branch-and-cut algorithm for the target visitation problem},
journal = {EURO Journal on Computational Optimization},
volume = {7},
number = {3},
pages = {209--242},
year = {2019},
issn = {2192-4406},
doi = {10.1007/s13675-019-00111-x}
}

@article{cordeau_dial--ride_2007,
	title = {The dial-a-ride problem: models and algorithms},
	volume = {153},
	issn = {1572-9338},
	shorttitle = {The dial-a-ride problem},
	doi = {10.1007/s10479-007-0170-8},
	language = {en},
	number = {1},
	journal = {Annals of Operations Research},
	author = {Cordeau, Jean-François and Laporte, Gilbert},
	month = sep,
	year = {2007},
	pages = {29--46},
}

@article{molenbruch_typology_2017,
	title = {Typology and literature review for dial-a-ride problems},
	volume = {259},
	issn = {1572-9338},
	doi = {10.1007/s10479-017-2525-0},
	language = {en},
	number = {1},
	journal = {Annals of Operations Research},
	author = {Molenbruch, Yves and Braekers, Kris and Caris, An},
	month = dec,
	year = {2017},
	pages = {295--325},
}

@article{ho_survey_2018,
	title = {A survey of dial-a-ride problems: {Literature} review and recent developments},
	volume = {111},
	issn = {0191-2615},
	shorttitle = {A survey of dial-a-ride problems},
	doi = {10.1016/j.trb.2018.02.001},
	language = {en},
	urldate = {2023-04-14},
	journal = {Transportation Research Part B: Methodological},
	author = {Ho, Sin C. and Szeto, W. Y. and Kuo, Yong-Hong and Leung, Janny M. Y. and Petering, Matthew and Tou, Terence W. H.},
	month = may,
	year = {2018},
	pages = {395--421},
}

\newpage

\section*{Tables}
\begin{table}[H]
    \centering
    \caption{Summary of Parameters.}
    \begin{tabularx}{\textwidth}{ll}
        \toprule
        Notation & Definition \\
        \midrule
        $\Hset$ & sequence of $\sizeH$ stations \\
        $\Kset$ & set of $\sizeK$ homogeneous vehicles \\
        $\Pset$ & set of $\sizeP$ positions per vehicle \\
        $\Pasc$ & the set of positions in ascending direction \\
        $\Pdesc$ & the set of positions in descending directions \\
        $\Sset$ & set of $\sizeS$ stopping patterns\\
        $\Jset$ & index set of $\Sset, \set{1, \ldots, \sizeS}$\\
        $\Rset$ & set of $\sizeR$ requests \\
        $\Rasc$ & ascending requests\\
        $\Rdesc$ & descending requests\\
        $\Gset$ & set of overlapping requests \\
        $\Gdir$ & set of all overlapping request sets $\Gset$ in a direction $\dir$\\
        $\dir$ & direction, either $\asc$ (ascending) or $\desc$ (descending) \\
        $\Qmax$ & vehicle capacity\\
        $t_{h,h'} \geq 0$ & distance between $h,h' \in \Hset$ \\
        $o_r, d_r \in \Hset$ & origin and destination of a request $r \in \Rset$\\
        $l_j$ & length of $s_j$\\
        $\Left{h}{j}$ & 1 if $h$ is the lowest stop in $s_j$, else 0\\
        $\Right{h}{j}$ & 1 if $h$ is the highest stop in $s_j$, else 0\\
        $s_j(h)$ & 1 if $h$ is included in $s_j$, else 0\\
        $\iota_r$ & reward of $r$\\
        $\wpax, \wdist$ & objective function weights (transported requests, saved distance)\\
        \bottomrule
    \end{tabularx}
    \label{tab:params}
\end{table}

\newpage

\begin{table}[H]
    \centering
    \caption{Best incumbent and best bound (given as incumbent / bound) found after a given time. A dash indicates no solution was found. The best values found per instance are emboldened. Rows where no values were found by either approach are omitted.}\footnotesize
    \begin{tabular}{@{}l cccc cccc@{}}
        \toprule
        Inst. & \multicolumn{4}{c}{ALAEB} & \multicolumn{4}{c}{Branch-and-Price}\\\cmidrule(lr){1-1}\cmidrule(lr){2-5}\cmidrule(lr){6-9}
        T [s] & 300 & 900 & 1800 & 3600 & 300 & 900 & 1800 & 3600 \\
        \midrule
        30-A & \textbf{361}/362.00 & \textbf{361}/361.65 & \textbf{361}/\textbf{361.00} & \textbf{361}/\textbf{361.00} & 359/366.20 & 359/365.83 & 360/364.84 & 360/364.41\\
        30-B & \textbf{359}/362.52 & \textbf{359}/362.52 & \textbf{359}/362.48 & \textbf{359}/\textbf{362.25} & 358/366.49 & 358/366.49 & 358/366.49 & \textbf{359}/366.49\\
        30-C & 346/350.34 & 348/350.34 & \textbf{349}/350.33 & \textbf{349}/\textbf{349.08} & 347/353.55 & 348/352.00 & 348/352.00 & 348/351.02\\
        30-D & 358/360.09 & \textbf{359}/359.88 & \textbf{359}/359.81 & \textbf{359}/\textbf{359.49} & 357/364.56 & 357/364.00 & 357/362.16 & 357/361.21\\
        30-E & 364/367.00 & \textbf{365}/367.00 & \textbf{365}/366.94 & \textbf{365}/\textbf{366.83} & 364/371.71 & \textbf{365}/371.71 & \textbf{365}/371.33 & \textbf{365}/370.35\\
        \rowcolor{PKlightgray!50}35-A & 410/\textbf{413.27} & 411/\textbf{413.27} & 411/\textbf{413.27} & \textbf{412}/\textbf{413.27} & 407/419.00 & 409/417.32 & 409/417.26 & 409/415.71\\
        \rowcolor{PKlightgray!50}35-B & 10/435.69 & 95/435.68 & \textbf{433}/435.68 & \textbf{433}/\textbf{435.52} & 430/447.00 & 431/444.60 & 431/444.00 & 431/444.00\\
        \rowcolor{PKlightgray!50}35-C & \textbf{405}/407.17 & \textbf{405}/407.17 & \textbf{405}/407.17 & \textbf{405}/\textbf{406.67} & 401/412.00 & 401/412.00 & 401/412.00 & 403/411.46\\
        \rowcolor{PKlightgray!50}35-D & \textbf{416}/\textbf{417.77} & \textbf{416}/\textbf{417.77} & \textbf{416}/\textbf{417.77} & \textbf{416}/\textbf{417.77} & 413/423.50 & 413/423.50 & 414/423.50 & 415/423.00\\
        \rowcolor{PKlightgray!50}35-E & \textbf{404}/404.67 & \textbf{404}/404.58 & \textbf{404}/404.55 & \textbf{404}/\textbf{404.33} & 403/407.00 & 403/407.00 & 403/407.00 & 403/407.00\\
        40-A & -/- & 0/\textbf{468.48} & 463/\textbf{468.48} & \textbf{468}/\textbf{468.48} & -/- & 463/477.00 & 463/477.00 & 463/477.00\\
        40-B & 464/\textbf{484.48} & \textbf{481}/\textbf{484.48} & \textbf{481}/\textbf{484.48} & \textbf{481}/\textbf{484.48} & 476/493.00 & 477/493.00 & 477/493.00 & 478/492.33\\
        40-C & 191/\textbf{474.68} & \textbf{473}/\textbf{474.68} & \textbf{473}/\textbf{474.68} & \textbf{473}/\textbf{474.68} & -/- & 471/483.00 & 471/483.00 & 471/483.00\\
        40-D & 453/\textbf{455.10} & \textbf{454}/\textbf{455.10} & \textbf{454}/\textbf{455.10} & \textbf{454}/\textbf{455.10} & 450/460.00 & 451/460.00 & 452/459.46 & 452/459.46\\
        40-E & 473/477.62 & \textbf{475}/\textbf{477.58} & \textbf{475}/\textbf{477.58} & \textbf{475}/\textbf{477.58} & 471/486.00 & 471/486.00 & 471/486.00 & 471/486.00\\
        \rowcolor{PKlightgray!50}45-A & -/- & 0/\textbf{537.14} & 0/\textbf{537.14} & 0/\textbf{537.14} & 530/547.00 & 532/547.00 & 532/546.31 & \textbf{533}/546.31\\
        \rowcolor{PKlightgray!50}45-B & -/- & 364/\textbf{552.01} & 506/\textbf{552.01} & 549/\textbf{552.01} & 547/564.00 & 547/564.00 & \textbf{548}/562.10 & \textbf{548}/562.10\\
        \rowcolor{PKlightgray!50}45-C & -/- & -/- & 0/\textbf{540.31} & 0/\textbf{540.31} & -/- & 534/552.00 & \textbf{535}/552.00 & \textbf{535}/550.40\\
        \rowcolor{PKlightgray!50}45-D & -/- & -/- & -/- & 0/\textbf{550.64 }& 540/563.00 & 546/563.00 & 546/563.00 & \textbf{547}/563.00\\
        \rowcolor{PKlightgray!50}45-E & \textbf{528}/530.07 & \textbf{528}/530.07 &\textbf{ }528/\textbf{530.04} & \textbf{528}/\textbf{530.04} & 526/538.00 & \textbf{528}/538.00 & \textbf{528}/537.20 & 528/537.20\\
        50-A & -/- & -/- & -/- & -/- & -/- & 600/\textbf{627.00} & 601/\textbf{627.00} & \textbf{602}/\textbf{627.00}\\
        50-B & -/- & 158/\textbf{593.54} & 188/\textbf{593.54} & 310/\textbf{593.54} & 584/604.00 & 584/604.00 & 584/604.00 & \textbf{587}/604.00\\
        50-C & -/- & -/- & 292/\textbf{602.19} & 292/\textbf{602.19} & -/- & \textbf{593}/616.00 & \textbf{593}/616.00 & \textbf{593}/616.00\\
        50-D & 395/\textbf{589.00} & 581/\textbf{589.00} & 581/\textbf{589.00} & 581/\textbf{589.00} & 581/599.00 & 584/599.00 & \textbf{585}/597.92 & \textbf{585}/597.92\\
        50-E & -/- & -/- & -/- & 158/\textbf{602.16} & -/- & \textbf{596}/614.00 & \textbf{596}/613.35 & \textbf{596}/613.35\\
        \rowcolor{PKlightgray!50}55-A & -/- & 215/\textbf{652.22} & 215/\textbf{652.22} & 277/\textbf{652.22} & -/- & \textbf{642}/666.00 & \textbf{642}/666.00 & \textbf{642}/666.00\\
        \rowcolor{PKlightgray!50}55-B & -/- & -/- & -/- & -/- & -/- & 660/\textbf{686.00} & 660/\textbf{686.00} & \textbf{661}/\textbf{686.00}\\
        \rowcolor{PKlightgray!50}55-C & -/- & -/- & -/- & -/- & -/- & \textbf{669}/\textbf{699.00} & \textbf{669}/\textbf{699.00} & \textbf{669}/\textbf{699.00}\\
        \rowcolor{PKlightgray!50}55-D & -/- & -/- & -/- & -/- & -/- & -/- & -/- & \textbf{655}/\textbf{681.00}\\
        \rowcolor{PKlightgray!50}55-E & -/- & 0/\textbf{646.18} & 0/\textbf{646.18} & 0/\textbf{646.18} & -/- & \textbf{637}/659.00 & \textbf{637}/659.00 & \textbf{637}/659.00\\
        60-C & -/- & -/- & -/- & -/- & -/- & -/- & \textbf{704}/\textbf{731.00} & \textbf{704}/\textbf{731.00}\\
        60-E & -/- & -/- & -/- & -/- & -/- & -/- & -/- & \textbf{717}/\textbf{746.00}\\
        \bottomrule
    \end{tabular}
    \label{tab:models:ub}
\end{table}

\newpage
\begin{table}[H]
    \centering
    \caption{Runtime in seconds of the restricted master problem when removing specific symmetry-breaking constraints. A dash indicates no solution was found within the time limit of 900 seconds. Rows where no solution was found in all columns are omitted.}
    \label{tab:symmconstr:runtime}
    \footnotesize
    \begin{tabularx}{.6\textwidth}{XSSSS}
        \toprule
        {Instance} & {Baseline} & {w/o \eqref{eq:master-milp:13}} & {w/o \eqref{eq:master-milp:14}} & {w/o \eqref{eq:master-milp:13} and \eqref{eq:master-milp:14}} \\
        \midrule
        {1-20-A} & 2.87 & 3.28 & 2.87 & 3.06\\
        {1-20-B} & 0.50 & 0.38 & 0.50 & 0.39\\
        {1-20-C} & 5.06 & 61.95 & 5.11 & 58.28\\
        {1-20-D} & 0.58 & 0.47 & 0.62 & 0.43\\
        {1-20-E} & 0.13 & 0.08 & 0.13 & 0.08\\
        \rowcolor{PKlightgray!50} {1-30-A} & 9.87 & 207.26 & 8.54 & 199.42\\
        \rowcolor{PKlightgray!50} {1-30-B} & 5.52 & 629.04 & 5.76 & 641.44\\
        \rowcolor{PKlightgray!50} {1-30-C} & 3.05 & 7.83 & 3.00 & 7.67\\
        \rowcolor{PKlightgray!50} {1-30-D} & 26.03 & 395.00 & 29.50 & 381.05\\
        \rowcolor{PKlightgray!50} {1-30-E} & 15.76 & 305.78 & 15.74 & 301.54\\
        {2-30-A} & 16.96 & 7.40 & 17.33 & 20.54\\
        {2-30-B} & 106.06 & {-} & 111.36 & {-}\\
        {2-30-C} & 9.02 & 118.22 & 7.06 & 30.17\\
        {2-30-D} & 64.24 & {-} & 120.78 & {-}\\
        {2-30-E} & 94.59 & {-} & 192.49 & 618.53\\
        \rowcolor{PKlightgray!50} {1-40-A} & 768.96 & {-} & 609.14 & {-}\\
        \rowcolor{PKlightgray!50} {1-40-B} & 109.35 & {-} & 107.19 & {-}\\
        \rowcolor{PKlightgray!50} {1-40-C} & 264.91 & {-} & 249.94 & {-}\\
        \rowcolor{PKlightgray!50} {1-40-D} & 72.35 & {-} & 70.38 & {-}\\
        \rowcolor{PKlightgray!50} {1-40-E} & 151.35 & {-} & 145.86 & {-}\\
        {2-40-A} & 661.88 & {-} & {-} & {-}\\
        {2-40-C} & 737.84 & {-} & 543.66 & {-}\\
        {2-40-D} & 255.13 & {-} & 284.55 & {-}\\
        {2-40-E} & 453.44 & {-} & 184.86 & {-}\\
        \rowcolor{PKlightgray!50} {2-50-B} & 816.37 & {-} & {-} & {-}\\
        \bottomrule
    \end{tabularx}
\end{table}

\newpage
\appendix
\section{Relaxed Restricted Master Problem}\label{sec:rmp}
For any set $\Jset'\subseteq \Jset$, the relaxed restricted master problem (RRMP($\Jset'$)) is given by
\begin{maxi!}[3]<b>
{}{\sum_{k \in \Kset} \Biggl( \sum_{\substack{\dir \in\\ \set{\asc, \desc}}} \sum_{p \in \Pdir} \sum_{r \in \Rdir} \left(\wpax + \wdist \cdot t_{o_r, d_r} \right) x_{p,k}^{r}  - \wdist \cdot d_k \Biggr)}{\protect\label{eq:rrmp}}{}
\addConstraint{\sum_{k \in \Kset} \sum_{p \in \Pdir} x_{p,k}^{r}}{\leq 1}{\quad \forall r \in \Rdir, \dirs}{\quad(\alpha_r)}{\protect\label{eq:rrmp:1-A}}
\addConstraint{\sum_{r \in \Gset} x_{p,k}^{r}}{\leq \Qmax}{\quad \forall p \in \Pdir, k \in \Kset, \Gset \in \Gdir, \dirs}{\quad(\phi_{p,k}^g)}{\protect\label{eq:rrmp:2-A}}
\addConstraint{\Sstart{h}{p,k} - \Send{h}{p-1,k}}{= 0}{\quad \forall p \in \Pset \setminus \set{1}, h \in \Hset, k \in \Kset}{\quad(\gamma_{p,k}^g)}{\protect\label{eq:rrmp:3-A}}
\addConstraint{\sum_{h \in \Hset} \Sstart{h}{p,k}}{= 1}{\quad \forall p \in \Pset, k \in \Kset}{\quad(\delta_{p,k})}{\protect\label{eq:rrmp:4-A}}
\addConstraint{\sum_{h \in \Hset} \Send{h}{p,k}}{= 1}{\quad \forall p \in \Pset, k \in \Kset}{\quad(\epsilon_{p,k})}{\protect\label{eq:rrmp:5-A}}
\addConstraint{\sum_{j \in \Jset'} y_{p,k}^j}{= 1}{\quad \forall p \in \Pset, k \in \Kset}{\quad(\zeta_{p,k})}{\protect\label{eq:rrmp:6-A}}
\addConstraint{x_{p,k}^{r} - \sum_{j \in \Jset'} {y_{p,k}^j \cdot s_j(o_r) \cdot s_j(d_r)}}{\leq 0}{\quad \forall r \in \Rdir, p \in \Pdir, k \in \Kset, \dirs }{\quad(\iota_{p,k}^r)}{\protect\label{eq:rrmp:7-A}}
\addConstraint{\sum_{j \in \Jset'} y_{p,k}^j \cdot \Left{h}{j} - \Sstart{h}{p,k}}{\leq 0}{\quad \forall p \in \Pasc, k \in \Kset, h \in \Hset}{\quad(\lambda_{p,k}^{h})}{\protect\label{eq:rrmp:8-A}}
\addConstraint{\sum_{j \in \Jset'} y_{p,k}^j \cdot \Right{h}{j} - \Sstart{h}{p,k}}{\leq 0}{\quad \forall p \in \Pdesc, k \in \Kset, h \in \Hset}{\quad(\lambda_{p,k}^{h})}{\protect\label{eq:rrmp:9-A}}
\addConstraint{\sum_{j \in \Jset'} y_{p,k}^j \cdot \Right{h}{j} - \Send{h}{p,k}}{\leq 0}{\quad \forall p \in \Pasc, k \in \Kset, h \in \Hset}{\quad(\mu_{p,k}^{h})}{\protect\label{eq:rrmp:10-A}}
\addConstraint{\sum_{j \in \Jset'} y_{p,k}^j \cdot \Left{h}{j} - \Send{h}{p,k}}{\leq 0}{\quad \forall p \in \Pdesc, k \in \Kset, h \in \Hset}{\quad(\mu_{p,k}^{h})}{\protect\label{eq:rrmp:11-A}}
\addConstraint{\sum_{p \in \Pset} \sum_{j \in \Jset'} l_j \cdot y_{p,k}^j - d_k}{\leq 0}{\quad \forall k \in \Kset}{\quad (\nu_k)}{\protect\label{eq:rrmp:12-A}}
\addConstraint{y_{p-2, k}^j + y_{p-1,k}^j - y_{p,k}^j}{\leq 1}{\quad\forall k \in \Kset, p \in \Pset\setminus \set{1, 2}, j \in \Jset' : l_j = 0}{\quad(\psi_{p,k}^j) }{\protect\label{eq:rrmp:13-A}}
\displaybreak
\addConstraint{{d_{k+1} - d_k}}{{\leq 0}}{\quad {\forall k \in \Kset \setminus \set{\sizeK}}}{\quad(\vartheta_{k})}{\protect\label{eq:rrmp:14-A}}
\addConstraint{d_k}{\geq 0}{\quad \forall k \in \Kset}{\protect\label{eq:rrmp:19-A}}
\addConstraint{x_{p,k}^{r}}{\geq 0}{\quad \forall r \in \Rdir, p \in \Pdir, k \in \Kset, \dirs}
\addConstraint{y_{p,k}^j}{\geq 0}{\quad \forall j \in \Jset', p \in \Pset, k \in \Kset}
\addConstraint{\Sstart{h}{p,k}}{\geq 0}{\quad \forall p \in \Pset, k \in \Kset, h \in \Hset}
\addConstraint{\Send{h}{p,k}}{\geq 0}{\quad \forall p \in \Pset, k \in \Kset, h \in \Hset}
\end{maxi!}
Note that there are no constraints needed to restrict the individual variables to be smaller or equal than one due to Constraints~\eqref{eq:rrmp:1-A}, \eqref{eq:rrmp:6-A}, \eqref{eq:rrmp:4-A}, and~\eqref{eq:rrmp:5-A} in conjunction with the nonnegativity constraints. For $\Jset'=\Jset$, we obtain the LP-relaxation (LP) of the (unrestricted) master problem.

The variables in parentheses behind each constraint correspond to the dual variables in (DP).

\newpage
\section{Dual Problem to the Relaxed Restricted Master Problem}\label{sec:dp}
For any set $\Jset' \subseteq \Jset$, the dual (DP($\Jset'$)) to (RRMP($\Jset'$)) is given by:

\begin{mini!}[3]<b>
{}{\sum_{r \in \Rasc} \alpha_r^{\asc} + \sum_{r \in \Rdesc} \alpha_r^{\desc}
    + \sum_{p \in \Pset}\sum_{k \in \Kset} \left(\delta_{p,k} + \epsilon_{p,k} + \zeta_{p,k} \right)
    + \sum_{p \in \Pset} \sum_{k \in \Kset} \sum_{\substack{j \in \Jset: \\ l_j = 0}} \psi_{p,k}^j}{\protect\label{eq:dp}}{}
    \breakObjective{+ \Qmax \sum_{g \in \Gasc} \sum_{p \in \Pasc} \sum_{k \in \Kset} \phi_{p,k}^{\asc, g} + \Qmax \sum_{g \in \Gdesc} \sum_{p \in \Pdesc} \sum_{k \in \Kset} \phi_{p,k}^{\desc, g} }\nonumber
\addConstraint{
    \alpha_r^{\dir} + \iota_{p,k}^r + \sum_{\substack{g \in \Gdir: \\ r \in g}}\phi_{p,k}^g}{\geq \wpax + \wdist \cdot t_{o_r, d_r}  \qquad \forall r \in \Rdir, p \in \Pdir, k \in \Kset, \dirs \protect\label{eq:dp:1}}
\addConstraint{\zeta_{p,k} - \sum_{r \in \Rasc} s_j(o_r) \cdot s_j(d_r) \cdot \iota_{p,k}^{r}}{\protect\label{eq:dp:2}}
\addConstraint{ + \sum_{h \in \Hset} \left( \Left{h}{j} \cdot \lambda_{p,k}^{h} + \Right{h}{j} \cdot \mu_{p,k}^{h} \right)+ l_j\nu_k}{\geq 0 \qquad\qquad\qquad \forall j \in \Jset: l_j > 0, p \in \Pasc, k \in \Kset}\nonumber
\addConstraint{\zeta_{p,k} - \sum_{r \in \Rdesc} s_j(o_r) \cdot s_j(d_r) \cdot \iota_{p,k}^{r}}{\protect\label{eq:dp:3}}
\addConstraint{+ \sum_{h \in \Hset} \left( \Right{h}{j} \cdot \lambda_{p,k}^{h} + \Left{h}{j} \cdot \mu_{p,k}^{h} \right) + l_j\nu_k}{\geq 0 \qquad\qquad\qquad \forall j \in \Jset : l_j > 0, p \in \Pdesc, k \in \Kset}\nonumber
\addConstraint{\zeta_{p,k} - \sum_{r \in \Rasc} s_j(o_r) \cdot s_j(d_r) \cdot \iota_{p,k}^{r} - \psi_{p,k}^j + \psi_{p+1,k}^j}{\protect\label{eq:dp:4}}
\addConstraint{+ \psi_{p+2,k}^j + \sum_{h \in \Hset} \left( \Left{h}{j} \cdot \lambda_{p,k}^{h} + \Right{h}{j} \cdot \mu_{p,k}^{h} \right) + l_j\nu_k}{\geq 0 \qquad\forall j \in \Jset: l_j = 0, p \in \Pasc \setminus \set{\sizeP - 1}, k \in \Kset}\nonumber
\addConstraint{\zeta_{p,k} - \sum_{r \in \Rdesc} s_j(o_r) \cdot s_j(d_r) \cdot \iota_{p,k}^{r} - \psi_{p,k}^j + \psi_{p+1,k}^j}{\protect\label{eq:dp:5}}
\addConstraint{+ \psi_{p+2,k}^j + \sum_{h \in \Hset} \left( \Right{h}{j} \cdot \lambda_{p,k}^{h} + \Left{h}{j} \cdot \mu_{p,k}^{h} \right) + l_j\nu_k}{\geq 0\qquad \forall j \in \Jset: l_j = 0, p \in \Pdesc \setminus \set{\sizeP}, k \in \Kset}\nonumber
\addConstraint{\gamma_{p,k}^h - \lambda_{p,k}^{h} + \delta_{p,k}}{\geq 0 \hspace{14.5em} \forall h \in \Hset, p \in \Pset \setminus \set{1}, k \in \Kset \protect\label{eq:dp:6}}
\addConstraint{- \lambda_{1,k}^{h} + \delta_{1,k}}{\geq 0 \hspace{16em} \forall h \in \Hset, k \in \Kset \protect\label{eq:dp:7}}
\addConstraint{-\gamma_{p+1,k}^h - \mu_{p,k}^{h} + \epsilon_{p,k}}{\geq 0 \hspace{12.2em} \forall h \in \Hset, p \in \Pset \setminus \set{1, \sizeP}, k \in \Kset \protect\label{eq:dp:8}}
\addConstraint{-\mu_{\sizeP,k}^{h} + \epsilon_{\sizeP,k}}{\geq 0 \hspace{16em} \forall h \in \Hset, k \in \Kset \protect\label{eq:dp:9}}
\addConstraint{-\mu_{1,k}^{h} + \epsilon_{1,k}}{\geq 0 \hspace{16em} \forall h \in \Hset, k \in \Kset \protect\label{eq:dp:10}}
\addConstraint{-\nu_k + {\vartheta_{k-1} - \vartheta_k}}{\geq -\wdist \hspace{11.5em} \forall k \in \Kset \setminus \set{1} \protect\label{eq:dp:11}}
\addConstraint{-\nu_1 - {\vartheta_1}}{\geq -\wdist\protect\label{eq:dp:12}}
\addConstraint{- \sum_{p \in \Pset}\sum_{k \in \Kset} \xi_{p,k}^{j}}{\geq 0 \hspace{16em} \forall j \in \Jset \protect\label{eq:dp:13}}
\addConstraint{\alpha_r}{\geq 0\hspace{21.2em}\forall r \in R \protect\label{eq:dp:14}}
\addConstraint{\beta_{p,k}^r, \iota_{p,k}^r}{\geq 0 \hspace{18.8em} \forall r \in R, p \in \Pset, k \in \Kset \protect\label{eq:dp:15}}
\addConstraint{\upsilon_{p,k}}{\geq 0 \hspace{20.8em} \forall p \in \Pset, k \in \Kset \protect\label{eq:dp:17}}
\addConstraint{\nu_k, \vartheta_k}{\geq 0 \hspace{20em} \forall k \in \Kset \protect\label{eq:dp:18}}
\addConstraint{\lambda_{p,k}^{h}, \mu_{p,k}^{h}}{\geq 0 \hspace{18.9em} \forall p \in \Pset, k \in \Kset, h \in \Hset \protect\label{eq:dp:19}}
\addConstraint{\phi_{p,k}^g}{\geq 0 \hspace{21em} \forall g \in \Gset, \Gset \in \Gdir,\dir, p \in \Pset, k \in \Kset \protect\label{eq:dp:20}}
\addConstraint{\psi_{p,k}^j}{\geq 0 \hspace{21em} \forall p \in \Pset\setminus \set{\sizeP}, k \in \Kset, j \in \Jset: l_j = 0 \protect\label{eq:dp:21}}
\end{mini!}
\noindent where the variables $\gamma_{p,k}^h, \delta_{p,k}, \epsilon_{p,k}, \zeta_{p,k}$ are unrestricted.
For $\Jset'=\Jset$ we obtain the dual of the unrestricted relaxed master problem.

\newpage
\section{Complexity of the \lipdp}\label{sec:appendix:complexity}
The decision version of the \lipdp is the following:
\defdecproblem{\lipdp}
{An instance $\Iset$ with stations $\Hset$ with distances $t$, requests $\Rset$, and vehicles $\Kset$ of capacity $\Qmax$, objective weights $\wpax$ and $\wdist$, an integer $k$}
{Does there exist a solution (a set of tours and set of accepted requests) with objective value at least $k$?}

We prove \Cref{thm:complexity:lipdp} by a reduction from the open traveling salesperson problem, which has been shown to be \NP-complete for Euclidean distances in \citep{PAPADIMITRIOU1977237}, where it is called the \textsc{path-TSP}.

\defdecproblem{\textsc{path-TSP}}
{A complete graph $G$ with Euclidean weights $w$, an integer $k$}
{Does $G$ admit a Hamiltonian path of length at most $k$?}

\begin{proof}
    A candidate solution to the \lipdp consists of a tour (sequence of stations) for each vehicle, whose travel costs can be verified in polynomial time, and a set of accepted requests.
    Together, the objective function value can be verified in polynomial time by summing the travel costs and request rewards. Thus, \lipdp is in \NP.

    Let $(G,k)$ be an instance of the \textsc{path-TSP}, where $G$ has $n$ nodes and weights $w_{i,j}\ \forall i,j \in \set{1, \ldots, n}$. 
    We construct an instance of the \lipdp as follows: a set of $n$ stations $\Hset = (1, \ldots, n)$ with distances $t_{i,j} = w_{i,j}\ \forall i,j \in \set{1, \ldots, n}$, a single vehicle of capacity $1$, and a set of $n$ unique requests $\Rset = \set{r_1, \ldots, r_n}$ where request $r_i$ has both its origin and destination at $h_i \in \Hset$. We set the objective function weights as $\wpax = \sum_{i,j = 1}^n{t_{i,j}}$ and $\wdist = 1$ and $k' = \wpax \cdot n - \wdist \cdot k$. Note that here $t_{o_r, d_r} = 0$ for all $r \in \Rset$. This construction can be executed in polynomial time.

    Suppose there exists a tour of length at most $k$ to the \textsc{path-TSP} on the graph $G$. The corresponding solution to the \lipdp corresponds to visiting (e.g., picking up and dropping off) each request exactly in the order specified by the \textsc{path-TSP} solution. Since this is a Hamiltonian path, every request is served. Then, the objective function value is given by
    \[\wpax \cdot n + \wdist \cdot (0 - k) =  \wpax \cdot n - \wdist \cdot k = k'. \]
    Thus there exists a solution to the \lipdp with objective value at least $k'$.

    Now, consider the \lipdp solution with the maximal objective value~$k^*$ and suppose that~$k^* \geq k'$.

    First, we argue that this solution visits all $n$ requests. If not, by 
    visiting an additional (previously unvisited) request $m$, the objective value changes by $\wpax - \wdist \cdot (t_{i,m} + t_{m,j} - t_{i,j})$, if this visit occurs between visiting requests $i$ and $j$, or by $\wpax - \wdist \cdot t_{i,m} = \wpax - \wdist \cdot t_{m,i}$ if the visit is inserted at the beginning or end of the tour. 
    By our choice of $\wpax$ and $\wdist$, we have
    \[ \wpax - \wdist \cdot (t_{i,m} + t_{m,j} - t_{i,j}) = \sum_{i,j = 1}^n{t_{i,j}} - (t_{i,m} + t_{m,j} - t_{i,j}) > 0. \]
    Then $k^*$ would not be maximal, which is a contradiction. So any optimal solution must serve all requests. 
    
    Second, we argue that this solution visits each station in $\Hset$ exactly once. If not, it must visit at least one station more than once. Then, only one of these visits corresponds to the pick-up and drop-off of the associated request (due to the fixed capacity of $1$). Since all distances respect the triangle inequality, omitting the additional visits to this station cannot worsen the solution (i.e., lengthen the vehicle tour). Hence there exists a solution of equivalent (or better) objective value which visits each station in $\Hset$ exactly once.
    
    Then, the \lipdp tour is w.l.o.g. a Hamiltonian path. Let~$t^*$ be the tour length, then it holds that~$k^*=\wpax\cdot n-\wdist\cdot t^*$ and therefore that~$t^*=\wpax\cdot n-k^*\leq k$ as~$k^*\geq k'$.

    Since the \lipdp graph corresponds exactly to the graph $G$, the \lipdp path is a Hamilton path in $G$ whose length is given by~$t^* \leq k$.

    Hence, the decision version of \lipdp is \NP-complete.
\end{proof}

\newpage
\section{The Impact of the Number of Stopping Patters}\label{sec:appendix:stopping_patterns}
In this experiment, we randomly generated a fixed a number of stopping patterns $\Sset'$, then solved the restricted master problem using $\Sset'$ with a time limit of 900 seconds. \Cref{fig:appendix:random_patterns:runtime} shows the runtime in seconds for different number of stopping patterns, with five samples per instance size. 

We observe that, at most, instances with 50 requests could be solved, where only the configuration with 100 stopping patterns was able to solve instances with over 40 requests. 

\begin{figure}[htb!]
    \centering
    \begin{tikzpicture}
\begin{axis}[%
    width=.85\linewidth, height=.3\linewidth,
    boxplot/draw direction=y,
    enlargelimits=0.05,
    xticklabels = {, , 100, 200, 300, 400, 500, 600, 700, 800, 900},
    xtick distance=1,
    xtick style={draw=none},
    ymin = 0,
    ymax=900,
    minor y tick num=1,
    xlabel={Number of Stopping Patterns},
    ylabel={Time [s]},
  ]
    \draw[PKlightgray,dashed]  (0,900) -- (10,900);

    \addplot+[thick,black,solid,boxplot prepared={draw position= 1,lower whisker=0.85497694, lower quartile=17.08630557, median=69.53702819, upper quartile=330.586874, upper whisker=870.4865971}] coordinates {};

    \addplot+[thick,black,solid,boxplot prepared={draw position= 2,lower whisker=0.633046436, lower quartile=18.68223946, median=151.487152, upper quartile=468.7938596, upper whisker=900}] coordinates {};
    
    \addplot+[thick,black,solid,boxplot prepared={draw position= 3,lower whisker=1.103090429, lower quartile=41.37236159, median=256.2953628, upper quartile=485.7702877, upper whisker=900}] coordinates {};

    \addplot+[thick,black,solid,boxplot prepared={draw position= 4,lower whisker=0.694281816, lower quartile=54.32140098, median=311.5625868, upper quartile=758.5873454, upper whisker=900}] coordinates {};

    \addplot+[thick,black,solid,boxplot prepared={draw position= 5,lower whisker=0.631036997, lower quartile=63.47300638, median=440.0936068, upper quartile=900, upper whisker=900}] coordinates {};

    \addplot+[thick,black,solid,boxplot prepared={draw position= 6,lower whisker=1.378111219, lower quartile=83.18446575, median=427.1953805, upper quartile=900, upper whisker=900}] coordinates {};

    \addplot+[thick,black,solid,boxplot prepared={draw position= 7,lower whisker=0.528181267, lower quartile=93.9648852, median=433.6791934, upper quartile=900, upper whisker=900}] coordinates {};

    \addplot+[thick,black,solid,boxplot prepared={draw position= 8,lower whisker=0.987705231, lower quartile=102.7539408, median=517.3767024, upper quartile=900, upper whisker=900}] coordinates {};

    \addplot+[thick,black,solid,boxplot prepared={draw position= 9,lower whisker=1.069824076, lower quartile=116.1661391, median=459.108209, upper quartile=900, upper whisker=900}] coordinates {};        
\end{axis}
\end{tikzpicture}
    \vspace{-.5em}
    \caption{Runtime of the restricted master problem for different number of randomly generated stopping patterns. The dashed line denotes the timeout at 900 seconds.}
    \label{fig:appendix:random_patterns:runtime}
\end{figure}
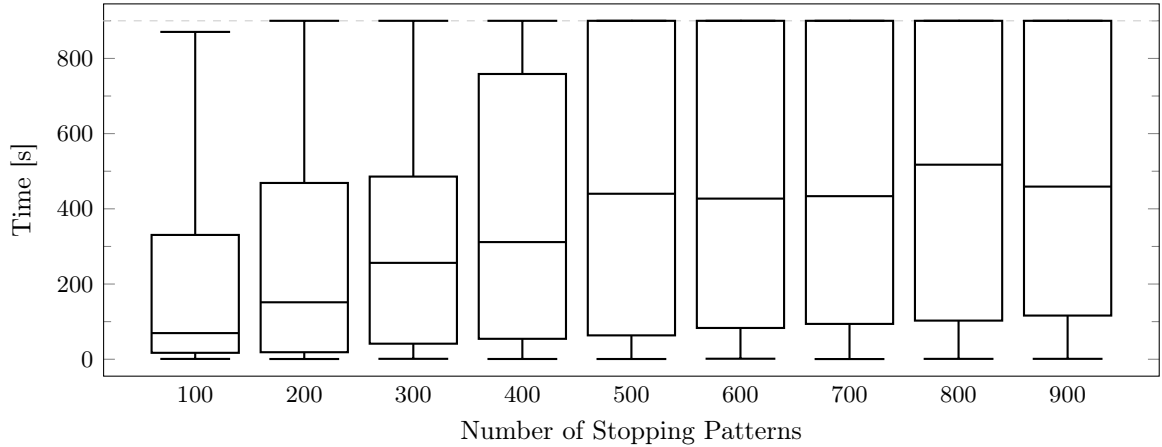
\end{document}